\documentclass[12pt]{article} 

\usepackage[utf8]{inputenc} 
\usepackage{geometry} 
\usepackage{graphicx} 
\usepackage[bb=boondox,bbscaled=.95,cal=boondoxo]{mathalfa}
\usepackage[all,cmtip]{xy}
\usepackage{booktabs} 
\usepackage{array} 
\usepackage{paralist} 
\usepackage{verbatim} 
\usepackage{subfig} 

\usepackage{authblk}
\usepackage{hyperref}
\usepackage{tikz}
\usepackage{tikz-cd}
\usepackage{tkz-euclide}
\usetikzlibrary{patterns}
\usetikzlibrary{calc}

\usepackage{amsfonts,amsmath,latexsym,amssymb} 
\usepackage{amsthm}            
\usepackage{mathptmx}
\usepackage{fancyhdr} 
\usepackage[all,cmtip]{xy}

\usepackage{authblk}
\usepackage{sectsty}
\allsectionsfont{\ttfamily\mdseries\upshape} 
\usepackage{titlesec}
\titleformat{\section}
  {\normalfont\sffamily\Large\bfseries\centering}
  {\thesection}{1em}{}
\titleformat{\subsection}
  {\normalfont\sffamily\large\bfseries\centering}
  {\thesubsection}{1em}{}

\usepackage[english]{babel}
\usepackage[nottoc,notlof,notlot]{tocbibind} 
\usepackage[titles,subfigure]{tocloft} 

\newtheorem{thm}{\textsc{Theorem}}[section]
\newtheorem*{thm*}{\textsc{Theorem}}
\newtheorem{cor}[thm]{\textsc{Corollary}}
\newtheorem{lem}[thm]{\textsc{Lemma}}
\newtheorem{prop}[thm]{\textsc{Proposition}}
\newtheorem*{prop*}{Proposition}
\newtheorem*{lem*}{\textsc{Lemma}}

\newtheorem*{rem*}{\textsc{Remark}}
\newtheorem{rem}[thm]{\textsc{Remark}}
\newtheorem{question}[thm]{\textsc{Question}}
\newtheorem{example}[thm]{\textsc{Example}}

\theoremstyle{definition}
\newtheorem*{exa*}{Example}
\newtheorem*{defn}{\textsc{Definition}}

\theoremstyle{remark}

\DeclareMathOperator{\spec}{spec}

\DeclareMathOperator{\h}{\mathcal{H}}
\DeclareMathOperator{\K}{\mathcal{K}}
\newcommand{\proposition}{\textsc{Proposition}}
\newcommand{\corollary}{\textsc{Corollary}}
\newcommand{\theorem}{\textsc{Theorem}}
\newcommand{\lemma}{\textsc{Lemma}}
\newcommand{\remark}{\textsc{Remark}}

\title{\bf\textsc{Beurling-Fourier Algebras and Complexification}}
\date{}
\author{Olof Giselsson \\\small Department of Mathematics\\\small Oslo University\\  \texttt{\footnotesize olofg@math.uio.no} \and Lyudmila Turowska\\ \small Department of Mathematical Sciences\\ \small Chalmers University of Technology and the University of Gothenburg\\ \texttt{\footnotesize turowska@chalmers.se} }
\begin{document}
\maketitle
\begin{abstract} 
In this paper, we develop a new approach that allows to identify the Gelfand spectrum of weighted Fourier algebras as a subset of an abstract complexification of the corresponding group for a wide class of groups and weights. This generalizes recent related results of Ghandehari-Lee-Ludwig-Spronk-Turowska \cite{gllst} about the spectrum of Beurling-Fourier algebras on some Lie groups. 
 In the case of discrete groups we show that the spectrum of Beurling-Fourier algebra is homeomorphic to the group itself. 
\end{abstract}

\vskip 1cm
\noindent
{\it 2000 Mathematics Subject Classification:} 46J15, 43A40, 43A30, 22D25.

\section{\sc{\textbf{Introduction}}}
Let $G$ be a locally compact group. The Fourier algebra $A(G)$, introduced by Eymard in \cite{Eym} , is a subalgebra of $C_0(G)$ consisting of the coefficients of the left regular representation $\lambda$ of $G$, i.e.
$$A(G)=\{u\in C_0(G)\;|\; u(s)=\langle\lambda(s)\xi,\eta\rangle,\xi,\eta\in L^2(G)\}.$$
The natural norm on $A(G)$ that makes it a Banach algebra is given by
$$\|u\|_{A(G)}=\inf\{\|\xi\|_2\|\eta\|_2\;|\; u(s)=\langle\lambda(s)\xi,\eta\rangle\},$$
where infimum is taken over all possible representations $u(s)=\langle\lambda(s)\xi,\eta\rangle$.
Moreover, $A(G)$ is the unique predual of the group von Neumann algebra $VN(G)$.

The Gelfand spectrum of the algebra is known to be topologically isomorphic to $G$, giving a non-trivial link between topological groups and Banach algebras. We note that the spectral theory has been an important tool in understanding commutative Banach algebras. 

Several authors, including the second author, have been investigating in~\cite{lee-samei, lst, ozrusp, glss, lss, gllst} a weighted version of the Fourier algebra, by imposing a weight that changes the norm structure.  Weighted Fourier algebras for compact quantum groups were studied in \cite{franz-lee}.  Recall that if $G$ is abelian with the dual group $\hat G$, the Fourier algebra $A(G)$ is isometrically isomorphic via the Fourier transform to $L^1(\hat G)$. If $w:\hat G\to [1,+\infty)$ is  a Borel measurable and  sub-multiplicative  function, i.e.
$$
\begin{array}{ccc}
w(st)\leq w(s)w(t),& s,t\in\hat G,
\end{array}
$$
(such $w$ is called a weight function) then $L^1(\hat G,w):=\{f\in L^1(G)\;|\; fw\in L^1(\hat G)\}$ is a subalgebra of $L^1(G)$ and is a Banach algebra with respect to the norm $\|f\|_w=\|fw\|_1$, $f\in L^1(\hat G,w)$.  Its image under inverse Fourier transform gives a weighted version $A(G,w)$ of $A(G)$. 
We note that for a weight function $w$ on $\hat G$ the (unbounded) operator
$$\tilde w=\int_{\hat G}^\oplus w(s)ds$$
defines a closed positive operator affiliated with $L^\infty(\hat G)\simeq VN(G)$  which  satisfies
$\Gamma(\tilde w)\leq \tilde w\otimes \tilde w$, where $\Gamma$ is the comultiplication on $VN(G)$.
This model has been taken in~\cite{lee-samei} and~\cite{gllst} to generalize the notion of weight to general locally compact groups. Accordingly, a weight, called a weight on the dual of $G$,  is a certain unbounded positive operator $\tilde w$ affiliated with $VN(G)$, which,  if in addition  $\tilde w$ is bounded below, i.e. $\omega:=\tilde w^{-1}\in VN(G)$,  satisfies
$$\omega\otimes\omega=\Gamma(\omega)\Omega$$
for a contractive 2-cocycle $\Omega\in VN(G\times G)$ (see \cite{gllst}). One can find numerous examples of non-trivial weights for general compact groups in \cite{lst} and certain connected Lie groups in \cite{gllst}. 
\\

In this paper, we will work with such weight inverse $\omega$ omitting the condition of its positivity. To each $\omega$ we will associate a subspace $A(G,\omega)$  of $A(G)$ which becomes a commutative Banach algebra with respect to a new weighted norm and the pointwise multiplication and so it is natural to study its Gelfand spectrum, $\spec\, A(G,\omega)$. When $G$ is compact and $\omega$ is a positive central weight, this question was studied in \cite{lst}; specific connected Lie groups, namely  $SU(N)$, the Heisenberg group, the reduced Heisenberg group, the Euclidean motion group $E(2)$ and its simply connected cover, were treated in the long paper \cite{gllst}. It has been proved that $\spec\,A(G,\omega)$ is closely related to an (abstract) complexification of $G$. To establish this fact the strategy  in \cite{gllst} was to find a simpler dense subalgebra $\mathcal A$ so that one could easily identify its spectrum, $\spec\,\mathcal A$, and get $\spec\,A(G,\omega)\subset\spec\, \mathcal A$. If $G$ is compact a natural choice is $\mathcal A=\text{Trig}\; G$, the algebra of matrix coefficients of finite-dimensional representations of $G$; $\spec\,\mathcal A$ is then an abstract complexification of $G$, introduced by McKennon in \cite{McK1}, which coincides in the case of compact connected Lie groups with the universal complexification of $G$. For non-compact groups, it seems there is no such natural choice of the subalgebra. In  \cite{gllst} the construction of $\mathcal A$ is rather technical and each $G$ treated in the paper required an individual approach, which heavily involved in particular the theory of group representations and technique of analytic extensions; the technicalities were an obstacle to develop a general theory applicable to any connected Lie group.

In this paper, we propose a different approach to the problem of identifying the spectrum of $A(G,\omega)$ that allows us to realise $\spec\, A(G,\omega)$ as a subset of an abstract complexification of $G$ for a wide class of groups and  weights.

The key idea is the observation that, identifying the dual of $A(G,\omega)$ with $VN(G)$, any multiplicative linear functional corresponds to $\sigma\in VN(G)$  satisfying the same equation as the weight inverse $\omega$, i.e. $\sigma\otimes\sigma=\Gamma(\sigma)\Omega$ for the contractive 2-cocycle $\Omega\in VN(G\times G)$ associated with $\omega$. A simple formal calculation, which we could make to be rigorous under certain conditions,  gives the equality 
$S(\sigma)\sigma=S(\omega)\omega$, where $S$ is the antipode on $VN(G)$. That  allows us  to define a closed operator $T_\sigma$, affiliated with $VN(G)$ and satisfying
$\Gamma(T_\sigma)=T_\sigma\otimes T_\sigma$ (\theorem~\ref{mainthm}).
It is known that the set of all non-zero $T\in VN(G)$ with  $\Gamma(T)=T\otimes T$ coincides with $\lambda(G)=\{\lambda(s)\;|\; s\in G\}$, providing  the embedding of $G$ into the spectrum of $A(G,\omega)$ through the evaluation $u\mapsto u(s)=(\lambda(s),u)$, $s\in G$; the set  $G_{\mathbb C, \lambda}^+$ of all positive solutions $T\in\overline{VN(G)}$ of $\Gamma(T)=T\otimes T$ is the image of the Lie algebra $\Lambda$ of derivations
$$\Lambda=\{\alpha\in \overline{VN(G)}\;|\; \alpha^*=-\alpha,\; \Gamma(\alpha)=\alpha\otimes 1+1\otimes\alpha\}$$ under the exponential map $\alpha\mapsto \exp{(i\alpha)}$;
 $G_{\mathbb C, \lambda}:=\lambda(G)\cdot G_{\mathbb C, \lambda}^+$ is then the space of all solutions to $\Gamma(T)=T\otimes T$ in $\overline{VN(G)}$ (\proposition~\ref{complexi}). Here $\overline{VN(G)}$ is the set of unbounded operators affiliated with $VN(G).$ In many cases including connected compact  and some nilpotent Lie groups $G$, 
 $G_{\mathbb C, \lambda}=\lambda_{\mathbb C}(G^u_{\mathbb C})$, where $G_{\mathbb C}^u$ is the universal complexification of $G$ and $\lambda_{\mathbb C}$ is the extension of the left regular representation to $G_{\mathbb C}^u$.
\\

The paper is organised as follows. In Section $2$, we introduce the notion of a weight inverse $\omega$ on the dual of $G$ and use this to define the Beurling-Fourier algebra $A(G,\omega)$ as a subalgebra of $A(G)$ with a modified norm and identify its dual with $VN(G)$. As $VN(G)$ has the unique predual, we show that $A(G,\omega)$ is isometrically isomorphic to $A(G)$ with a modified product $\cdot_\Omega$, depending on the $2$-cocycle $\Omega$ associated with $\omega$ and not the particular weight $\omega$. In \proposition~\ref{2weights} we give a necessary and sufficient condition for the inclusion $A(G,\omega_1)\subset A(G,\omega_2)$.  
\\

In Section $3$, we review some basic concepts  on unbounded operators and operators affiliated with a von Neumann algebra, and define the $\lambda$-complexification $G_{\mathbb C,\lambda}$ of $G$ as the set of  non-zero (unbounded) closed operators $T$ which are affiliated with  $VN(G)$ and  satisfy the equation  $\Gamma(T)=T\otimes T.$
\\

In Section $4$ we investigate the relation that the $\lambda$-complexification has to  the Gel\-fand spectrum of $A(G,\omega)$. We prove the embedding of  $\spec\, A(G,\omega)$ into the complexification  $G_{\mathbb C,\lambda}$ for a wide class of groups and weights;  this comes down to verifying that $S(\sigma)\sigma=S(\omega)\omega$ holds for the points $\sigma$  in $\spec\, A(G,\omega),$ considered as a subset of $VN(G).$ We also give a heuristic reason why we conjecture that this holds in general. These arguments give immediately  the equality for any virtually abelian group and any weight considered on it. The other cases of $G$ and $\omega$ for which the embedding of $\spec\, A(G,\omega)$ into $G_{\mathbb C,\lambda}$ holds include, for example,   compact, discrete and more general [SIN]-groups with arbitrary weights and  general locally compact groups with weights extended from weights on the dual of abelian or compact subgroups.  Even though we could not establish the inclusion result in full generality our approach allows us to generalise most of the previous results and avoid the main technicalities in \cite{gllst} to find a dense subalgebra which plays the role of $\text{Trig }G$ for the compact case. Moreover, as the main available source of weights on the dual of non-commutative groups are the weights induced from abelian or compact subgroups, \theorem~\ref{very} and \theorem~\ref{veryext} cover most of the known Beurling-Fourier algebras.  For  discrete group $G$ we show that  the spectrum of the corresponding Beurling-Fourier algebra is homeomorphic to $G$. 

Finally, in Section $5$, we discuss some of the questions that arose during our investigation, as well as some examples that show the necessity of certain conditions.

\section{\sc{\textbf{Beurling-Fourier Algebras}}}\label{section2}
For a locally compact group $G,$ we let $\lambda:G\to B(L^2(G))$   be  the left regular representations on $L^{2}(G)$.  Let  $VN(G)\subseteq B(L^{2}(G))$ be the group von Neumann algebra, $C_{r}^*(G)\subseteq VN(G)$ the reduced group $C^{*}$-algebra and $W\in VN(G)\bar\otimes B(L^2(G))$ the fundamental  multiplicative unitary, implementing the co-multiplication $\Gamma:VN(G)\to VN(G)\bar \otimes VN(G)$ as
\begin{equation}\label{gamma}\Gamma(x)=W^{*}(I\otimes x)W.
\end{equation}
Recall that $\Gamma$ is the unique normal $*$-homomorphism satisfying $\Gamma(\lambda(s))=\lambda(s)\otimes\lambda(s)$, $s\in G$, and $W\in B(L^2(G\times G))$ is given by the action 

$$
\begin{array}{ccc}
(W\xi)(s,t)=\xi(ts,t),& \text{for $\xi\in L^2(G\times G).$}
\end{array}
$$
The coproduct $\Gamma$ is co-commutative and satisfies the co-associative law:
\begin{equation}\label{coass}
(\iota\otimes\Gamma)\circ \Gamma=(\Gamma\otimes\iota)\circ\Gamma,
\end{equation}
where $\iota $ is the identity map. 
\begin{defn}[\sc{Weight Inverse}]
A $\omega\in VN(G)$  will be called a \textit{weight inverse on the dual of $G$} (we usually abbreviate this to a \textit{weight inverse}) if
\begin{equation}\label{iw}
\omega\omega^{*}\otimes\omega\omega^{*}\leq \Gamma(\omega\omega^{*})
\end{equation}
and
\begin{equation}\label{kernel}
\ker \omega=\ker \omega^{*}=\{0\}.
\end{equation}
\end{defn}

We note that if
 $\omega\in VN(G)$ is a weight inverse, then $\ker\,\Gamma(\omega)=\ker\,\Gamma(\omega^*)=\{0\}$ which follows easily from (\ref{gamma}) and (\ref{kernel}).

\begin{lem}\label{lem1}
Let $\omega\in VN(G)$ be a weight inverse. Then
there exists an injective $\Omega\in VN(G)\bar \otimes VN(G)\simeq VN(G\times G)$ of norm $||\Omega||\leq 1,$ such that
\begin{equation}\label{weight}
\omega\otimes\omega=\Gamma(\omega)\Omega
\end{equation}
and
$\Omega$ satisfies the $2$-cocycle relation
\begin{equation}\label{2co}
(\iota\otimes \Gamma)(\Omega)(I\otimes \Omega)=(\Gamma\otimes \iota)(\Omega)(\Omega\otimes I),
\end{equation}
\end{lem}

\begin{proof}
$\ker\Gamma(\omega^*)=\{0\}$ and the inequality~\eqref{iw} give a well-defined linear map
$$
\begin{array}{cccc}
\Omega^{*}: \Gamma(\omega^{*})x\mapsto (\omega^{*}\otimes \omega^{*})x, & \text{for $x\in L^{2}(G)\otimes L^{2}(G),$}
\end{array}
$$
that satisfies $||\Omega^{*}(\Gamma(\omega^{*})x)||^{2}=||(\omega^{*}\otimes \omega^{*})x||^{2}\leq ||\Gamma(\omega^{*})x||^{2}$. As the range $\mathrm{Ran}(\Gamma(\omega^{*}))$ is dense in $L^{2}(G\times G)$, we can  extend  $\Omega^*$ to a bounded linear operator on the whole Hilbert space. Let $\Omega$ be its adjoint. Clearly $||\Omega||\leq 1$ and \eqref{weight} holds.

It is easy to see from~\eqref{kernel} and~\eqref{weight} that $\Omega$ must commute with any element in the commutant $(VN(G)\bar\otimes VN(G))'$  and thus $\Omega\in VN(G)\bar\otimes VN(G).$ By~\eqref{weight}, it follows that $\ker\,\Omega\subset\ker\,(\omega\otimes\omega)=\{0\}$ and therefore $\Omega$ is  injective. Using ~\eqref{weight}, we get
$$
(\iota\otimes\Gamma)(\Gamma(\omega))(\iota\otimes \Gamma)(\Omega)(I\otimes \Omega)=\omega\otimes\omega\otimes\omega=(\Gamma\otimes\iota)(\Gamma(\omega))(\Gamma\otimes \iota)(\Omega)(\Omega\otimes I).
$$
Finally the co-associativity of $\Gamma$ and \eqref{kernel} imply  \eqref{2co}.
\end{proof}

\begin{rem*}
\begin{enumerate}[$(i)$]
\item If $\omega\in VN(G)$ satisfies (\ref{weight})
then it satisfies (\ref{iw}): $$\omega\omega^*\otimes\omega\omega^*=\Gamma(\omega)\Omega\Omega^*\Gamma(\omega)^*\leq\Gamma(\omega\omega^*).$$
Therefore a weight inverse could be also defined as $\omega\in VN(G)$ satisfying (\ref{kernel}) and (\ref{weight}) instead.
 \item
    It follows from \eqref{iw} that
$$||\omega ||^{4}=||\omega\omega^{*}\otimes\omega\omega^{*}||\leq ||\Gamma(\omega\omega^{*})||=||\omega||^{2},$$ so that $||\omega||^{2}\leq 1$ and hence a weight inverse is always a contraction.
\item In \cite{gllst} a (bounded below) weight on the dual of $G$ was defined as  an (unbounded) {\it  positive} operator $w$ which is affiliated with $VN(G)$ and admits an inverse $w^{-1}\in VN(G)$ such that $\Gamma(w)(w^{-1}\otimes w^{-1})$ is defined and contractive on a dense subspace, i.e. $w^{-1}$ is a positive weight inverse, in our terminology.
\item A weight inverse was considered in~\cite{ozrusp} as an element in the multiplier algebra $M(C_{r}^{*}(G))$ of $C_r^*(G)$ satisfying some additional density conditions.  If $G$ is compact, $M(C_r^*(G))=VN(G)$ and our definition coincides with the one in \cite{ozrusp}. 
\item The notion of unitary dual 2-cocycle on a compact group was introduced by Landstad \cite{landstad} and Wassermann \cite{wassermann} in the study of ergodic actions. In the context of quantum groups it was defined by Drinfeld \cite{drinfeld}. Their 2-cocycle condition is similar and defined as follows:
\begin{equation}\label{2co_dif}
(I\otimes \Omega)(\iota\otimes \Gamma)(\Omega)=(\Omega\otimes I)(\Gamma\otimes \iota)(\Omega).
\end{equation}
The cocycle of the form $(u\otimes u)\Gamma(u)^{-1}$ is called a coboundary. 
The inverse of our 2-cocycle satisfies (\ref{2co_dif}).  

\end{enumerate}
\end{rem*}

\begin{example}\rm 
Let  $w$ be a bounded below weight function on $G=\mathbb R$ or $\mathbb Z$ given by $w(x)=e^{i\gamma x}(1+|x|)^\alpha$ or $w(x)=e^{i\gamma x+\beta|x|}$, $\alpha,\beta >0$, $\gamma\in\mathbb R$. It is s easy to check that 
$$
\begin{array}{cccc}
|w(x+y)|\leq |w(x)||w(y)|, & \text{for all $x,y\in \mathbb{R},$}
\end{array}
$$
and moreover that $w^{-1}(x)$ is bounded. As $VN(\mathbb R)\simeq L^\infty(\mathbb R)$ and $VN(\mathbb T)\simeq\ell^\infty(\mathbb Z)$ via the Fourier transform, the image $\omega$ of $w^{-1}$ is a weight inverse on the dual of $G$. 

The above  weight inverses can be extended to $VN(\mathbb R^k\times\mathbb T^{n-k})$ by tensoring: $\omega=\omega_1\otimes\ldots\otimes \omega_n$. 

If a group $G$ contains a closed subgroup $H$ isomorphic to $\mathbb R^k\times\mathbb T^{n-k}$ then any weight inverse $\omega$ on the dual of $\mathbb R^k\times\mathbb T^{n-k}$ can be lifted to $VN(G)$ by considering $\omega_G=\iota_H(\omega)$, where $\iota_H: VN(H)\to VN(G)$ is the injective homomorphism $\lambda_H(s)\mapsto\lambda_G(s)$, here $\lambda_G$ and $\lambda_H$ are the left regular representations of $G$ and $H$ respectively; the existence of $\iota_H$ is due to Herz restriction theorem, see for example \cite{herz}.
\end{example}

For other examples  of weights  and weight inverses, we refer the reader to \cite{gllst}.

\medskip

Let $A(G)$ be the unique pre-dual of $VN(G)$. Recall that it can be identified with the space of functions on $G$:  $$A(G)=\{g\ast\check h\;|\; g,h\in L^2(G)\}\subset C_0(G),$$ where $\check h(s)=h(s^{-1})$, $s\in G$, and 
$$g\ast\check h=\int g(t)h(s^{-1}t)dt=\langle\lambda(s)h,\bar g\rangle;$$
$A(G)$ becomes a 
 commutative Banach algebra, usually called the Fourier algebra of $G$, with respect to the pointwise multiplication and the norm given by
$$\|f\|_{A(G)}=\inf\|g\|_2\|h\|_2,$$ where the infimum is taken over all possible decomposition $f=g\ast\check h$, see for example  \cite{Eym, kaniuth-lau}. The duality between $VN(G)$ and $A(G)$ is given by
$$( T,u )=\langle T \xi,\eta\rangle$$ for $T\in VN(G)$ and $u(s)=\langle\lambda(s)\xi,\eta\rangle=(\bar \eta\ast\check \xi)(s)\in A(G);$ here and through the rest of the paper we use $\langle\cdot,\cdot\rangle$ to denote the inner product on a Hilbert space and we keep notation $(\cdot,\cdot)$ for duality pairing between $\mathcal M$ and $\mathcal M_*$ when $\mathcal M$ is a von Neumann algebra.  

For $T\in VN(G)$ and $f\in A(G),$ we let $Tf\in A(G)$ be given by
$$
\begin{array}{ccc}
( R,Tf):=( RT,f), &\text{for $R\in VN(G).$}
\end{array}
$$
The assignment $T,f\mapsto Tf$ turns $A(G)$ into a left $VN(G)$-module.

If $\omega$ is a weight inverse, we define 
$$A(G,\omega):=\omega A(G)=\{\omega f \;|\; f\in A(G)\}\subset A(G)$$
and call it the \textit{Beurling-Fourier algebra} of $G$ associated to $\omega$.  
\begin{prop}
$A(G,\omega)$ is a Banach algebra with respect to the pointwise multiplication and the norm 
$$||\omega f||_{\omega}:=||f||_{A(G)}.$$
Moreover, $A(G,\omega)$ is a predual of $VN(G)$ 
with the pairing given by
\begin{equation}\label{duality}(\cdot,\cdot)_{\omega}:VN( G)\times A( G,\omega)\to \mathbb{C},\end{equation}
 $$( T,\omega f)_{\omega}=( T,f).$$
\end{prop}
\begin{proof}
 To see that 
$||\cdot ||_{\omega}$  is a norm,
we should only see  that it is well defined. 
In fact, if $f=\bar\eta\ast\check\xi$ then 
$$
\begin{array}{ccc}
(\lambda(s),\omega f)=\langle\lambda(s)\omega\xi, \eta\rangle=\langle\omega\xi,\lambda(s^{-1})\eta\rangle, &\text{for $s\in G.$}
\end{array}
$$
Let ${\mathcal U}=\overline{[\lambda(s)\eta\;|\;s\in G]}$, the closed linear span of $ \lambda(s)\eta$ , $s\in G$, and let $P$ be the projection onto ${\mathcal U}$. As ${\mathcal U}$  is invariant with respect to $VN(G)$, we have $P\omega=\omega P$.  Assuming now that  $\omega f=0$, we obtain  $\langle\omega\xi,\lambda(s^{-1})\eta\rangle=0$ for any  $s\in G,$  and hence $\omega P\xi=P\omega\xi=0$. By \eqref{kernel}, $P\xi=0$ and hence $f(s)=\langle\xi,\lambda(s^{-1})\eta\rangle=0$ for any  $s\in G$.
From~\eqref{weight} it follows that $A(G, \omega)$ is a commutative Banach algebra;  in fact, we have
$$
\begin{array}{cccc}
(\omega u)(\omega v)=\omega(\Gamma_*(\Omega (u\otimes v))),& \text{for $ u, v\in A(G),$}
\end{array}
$$   and \begin{equation}\label{product}\|(\omega u)(\omega v)\|_{\omega}= \|\Gamma_*(\Omega(u\otimes v)))\|_{A(G)}\leq \|u\|_{A(G)}\|v\|_{A(G)}=\|\omega u\|_{\omega}\|\omega v\|_{\omega}, 
\end{equation}
where $\Gamma_*: A(G)\widehat\otimes A(G)\to A(G)$ is the predual of the co-multiplication $\Gamma$ defined on the operator space projective product of $A(G)\widehat\otimes A(G)$ (see \cite{effros-ruan}). The associativity of the product is clear  and the completeness follows from the boundedness of  $\omega$: $\{\omega f_{n}\}$ is Cauchy in $A(G,\omega)$ if and only if $f_{n}$ is Cauchy in $A(G),$ so that $f_{n}\to f$ for some $f\in A(G)$ and hence also $\omega f_{n}\to \omega f.$
That $A(G,\omega)$ is a  predual of $VN(G)$  is obvious. 
\end{proof}

We note that  the previous proposition was proved in \cite{gllst} for positive weight inverses. Similar arguments can be applied to prove the general case. For the reader's convenience,  we have chosen to give its full proof.

The next statement shows that we can restrict ourselves to positive weight inverses.
\begin{prop}\label{propo2}
If $\omega$ is a weight inverse and  $\omega^{*}=U|\omega^{*}|$ is the polar decomposition of $\omega^*$, then $|\omega^{*}|$ is a weight inverse and the identity map $\omega u\mapsto \omega u$, $u\in A(G)$, defines an isometric isomorphism $A(G,\omega)\to A(G,|\omega^{*}|)$.
\end{prop}
\begin{proof}
Note that $U\in VN(G)$ is unitary by~\eqref{kernel}. From~\eqref{weight} it is immediate that $|\omega^{*}|$ is again a weight inverse. Clearly $A(G,\omega)=A(G,|\omega^{*}|)$ as subsets of $A(G),$ and the identity is an algebra homomorphism. Moreover
$$
\| \omega u\|_{|\omega^{*}|}=\| |\omega^{*}|(U^{*} u)\|_{|\omega^{*}|}=\|U^{*}u \|_{A(G)}=\|u \|_{A(G)}=\|\omega u \|_{\omega}.
$$
\end{proof}
We will use the following lemma:
\begin{lem}\label{complem}
If $\mathcal M\subseteq B(\mathcal{H})$ is a von Neumann algebra and  $a_{1},a_{2}\in \mathcal M$ satisfy
$$
a_{1}\mathcal M_{*}\subseteq a_{2}\mathcal M_{*},
$$
then there is $c\in\mathcal  M$ such that $a_{1}=a_{2}c.$ Moreover, we can assume that $\ker \,c=\ker\,a_{1}$, and $\ker\, a_{2}\subseteq \ker \,c^{*}$, and under these assumptions, $c$ is uniquely determined. 
\end{lem}
\begin{proof}
Let $a_i=S_{i}|a_{i}|,$ for $i=1,2$, be the polar decompositions, and let $P_{i}=S_{i}^{*}S_{i},$ so that $a_{i}P_{i}=a_{i}.$ For $i=1,2,$ the maps $P_{i}\mathcal M_{*}\to a_{i}\mathcal M_{*},$ defined as $f\mapsto a_{i}f$, are bijective linear maps. As $a_{1}\mathcal M_{*}\subseteq a_{2}\mathcal M_{*},$ there is for every $f\in P_{1} \mathcal M_{*}$ a unique $h(f)\in P_{2}\mathcal M_{*}$ such that $a_{1}f=a_{2}h(f)$. Let $R(f)=h(f)$. Clearly $R$ is  a linear injective map and moreover for $b\in \mathcal M,$ we have $R(fb)=R(f)b.$ Note that $P_{i}\mathcal M_{*}$, $i=1,2$, are closed subspaces of $\mathcal M_*$, thus Banach spaces. We claim that $R$ is closed: let $f_{n}$ be a sequence such that $f_{n}\to f$ and $R(f_{n})\to h$ as $n\to \infty.$ Then 
$$a_1f=\lim_{n\to\infty}a_1f_{n}=\lim_{n\to \infty}a_2R(f_{n})=a_2h,$$
so that $R(f)=h.$ As $R$ is defined on the whole  $P_{1}\mathcal M_{*},$ it is thus bounded.
Extend $R$ to all of $\mathcal M_{*}\cong(I-P_{1})\mathcal M_{*}\oplus P_{1}\mathcal M_{*}$ by the formula $R(f)=R(P_{1}f)$. Clearly, for this extension we still have $a_{2}R(f)=a_{1}f$, as well as $R(fb)=R(f)b$ for all $b\in \mathcal M.$ Let $R':\mathcal M\to \mathcal M$ be the dual of $R$. Then as 
$$
(mR'(b),f )=(R'(b),f m)=(b,R(fm))=(b,R(f)m)
=(mb,R(f))=(R'(mb),f)
$$
for all $b,m\in \mathcal M$ and $f\in \mathcal M_{*}$,
it follows $R'(mb)=mR'(b).$ Thus with $c=R'(I)\in \mathcal M,$ we have $R'(b)=bc.$ We get
$$
\begin{array}{ccc}
(b,R(f))=(R'(b),f)=(bc,f)=(b,cf), & \text{for all $b\in \mathcal M$ and $f\in \mathcal M_{*}$,}
\end{array}
$$
so that $R(f)=cf.$ It gives $a_{1}f=a_{2}R(f)=a_{2}c f$, and thus $a_{1}=a_{2}c.$ Clearly, $\ker \,c=\ker\, a_{1}$, $\ker\, c^{*}\supseteq \ker\, I-P_{2}=(\ker\, a_{2})^{\bot}$ and that $c$ is the unique element such that $a_{1}=a_{2}c$ with these properties.  
\end{proof}

 \begin{prop}\label{2weights}
Let $\omega_{1} ,\omega_{2}$ be two weight inverses on the dual of $G$. The inclusion $A(G,\omega_{1})\subseteq A(G,\omega_{2})$  implies that there  is $a\in VN(G)$ such that $\omega_{1}=\omega_{2}a$. Furthermore, we have $A(G,\omega_{1})= A(G,\omega_{2})$ if and only if $\omega_{1}=\omega_{2}a$ for an invertible element $a\in VN(G)$.
\end{prop}
\begin{proof}
It follows from $\lemma$~\ref{complem} that if $A(G,\omega_{1})\subseteq A(G,\omega_{2})$, then there is an $a\in VN(G)$ such that $\omega_{1}=\omega_{2}a.$ Moreover, if actually $A(G,\omega_{1})= A(G,\omega_{2}),$ then we get $a,b\in VN(G)$ such that $\omega_{1}=\omega_{2}a$ and $\omega_{2}=\omega_{1}b$. It then follows that $\omega_{1}(I-ba)=0$ and $\omega_{2}(I-ab)=0$ and as $\ker\,\omega_{i}=\{0\}$ for $i=1,2,$ we get $ba=ab=I$, so that $a$ is invertible. 
\end{proof}

Another equivalent model of the Beurling-Fourier algebra, which was given in \cite{gllst} for positive weights, is defined as follows. For a weight inverse $\omega$ and the corresponding 2-cocycle $\Omega$ define a new multiplication on $A(G)$ by 
\begin{equation}\label{omegaproduct}
\begin{array}{ccc}
u\cdot_{\Omega} v=\Gamma_*(\Omega(u\otimes v)),&\text{for $u,v\in A(G).$}
\end{array}
\end{equation}
It follows from (\ref{product}) that $(A(G),\cdot_\Omega)$ is a commutative contractive Banach algebra which is isomorphic to $A(G,\omega)$, showing that $A(G,\omega)$ can be  determined by the 2-cocycle $\Omega$ rather than the weight inverse $\omega$.  
Assume  $A(G,\omega_1)=A(G,\omega_2)$  and let $a\in VN(G)$ be the invertible operator such that $\omega_1=\omega_2a$ which exists due to Proposition \ref{2weights}. If $\Omega_1$ and $\Omega_2$ are  the corresponding 2-cocycles, then 
\begin{equation}\label{cocycle}
\Gamma(a)\Omega_1=\Omega_2(a\otimes a)
\end{equation}
and $(A(G),\cdot_{\Omega_1})\simeq (A(G),\cdot_{\Omega_2})$. The converse also holds: if $\Omega_1$, $\Omega_2\in VN(G\times G)$ are  2-cocycles  that satisfy (\ref{cocycle})  and  correspond to weight inverses $\omega_1$ and $\omega_2$ respectively, then $u\mapsto a u$, $u\in A(G)$, gives the isometric isomorphism $(A(G), \cdot_{\Omega_1})\simeq (A(G),\cdot_{\Omega_2})$.  To see this let $u$, $v\in A(G)$ and $x\in VN(G)$. Then 
\begin{eqnarray*}
(x, a(u\cdot_{\Omega_1}v))&=&(xa, u\cdot_{\Omega_1}v)=(\Gamma(xa),\Omega_1(u\otimes v))=(\Gamma(x)\Gamma(a)\Omega_1,u\otimes v)\\&=&(\Gamma(x)\Omega_2(a\otimes a),u\otimes v)=(\Gamma(x),\Omega_2(au\otimes av))=(x,au\cdot_{\Omega_2}av).
\end{eqnarray*}
If $a$ is not assumed to be invertible, the map $u\mapsto au$ gives a homomorphism from $(A(G), \cdot_{\Omega_1})$ to $ (A(G),\cdot_{\Omega_2})$.  We note that any $2$-cocycle associated with a weight inverse is symmetric, that is  invariant under the 'flip' automorphism $a\otimes b\mapsto b\otimes a$ of $VN(G)\bar \otimes VN(G)$. 
\\

\begin{rem}\rm 
Let $\mathcal{Z}^{2}_{sym}(\hat G)$ be a category whose objects are injective symmetric $2$-cocycles and elements  $a\in\mathrm{Hom}(\Omega_{1},\Omega_{2})$ are the non-zero operators in $ VN(G)$ satisfying~\eqref{cocycle}. With this notation, each weight inverse  belongs to the set $\mathrm{Hom}(\Omega,I),$ where the identity operator $I$ is considered as the identity $2$-cocycle: in this case \eqref{cocycle} becomes
$$
\Gamma(a)\Omega=a\otimes a.
$$
In particular, we have $\mathrm{Hom}(I,I)=\lambda(G)\simeq G$. Defining the equivalence relation on $\mathcal{Z}^{2}_{sym}(\hat G)$ by $\Omega_{1}\sim \Omega_{2}$ if there is an invertible element $a\in \mathrm{Hom}(\Omega_{1},\Omega_{2}),$ we have the embedding of the set of Beurling-Fourier algebras into $\mathcal{Z}^{2}_{sym}(\hat G)/\sim$.  

\end{rem}

We finish this section by defining a representation of $(A(G),\cdot_\Omega)$. 

Recall the fundamental unitary $W\in VN(G)\bar\otimes B(L^2(G))$ and let $f\in A(G)$,  $f(\cdot)=\langle \lambda(\cdot)\xi,\eta\rangle$. Then for  $x$, $y\in L^2(G)$ we have
\begin{eqnarray*}
&&\langle( f\otimes \iota)(W)x,y\rangle=\langle W(\xi\otimes x),\eta\otimes y\rangle=\int_{G\times G}\xi(ts)x(t)\overline{\eta(s)}\overline{y(t)}dtds=\\
&&=\int_G(\int_G\xi(ts)\overline{\eta(s)}ds)x(t)\overline{y(t)}dt=\int_G\langle\lambda(t)^{-1}\xi,\eta\rangle x(t)\overline{y(t)}dt=\langle M_{\check f}x,y\rangle,
\end{eqnarray*}
where $\check f(t)=f(t^{-1})$ and $M_{\check f}$ is the multiplication operator by $\check f$.

For $X\in B(L^2(G)\otimes L^2(G))$  write $X_{12}=X\otimes I$, $X_{23}=I\otimes X$ for operators on $L^2(G)\otimes L^2(G)\otimes L^2(G)$ and define $X_{13}$ similarly. Then $W$ satisfies the pentagonal relation 
\begin{equation}\label{pentagon}
W_{23}W_{12}=W_{12}W_{13}W_{23}.
\end{equation} 

For $f\in A(G)$ define $\lambda_\Omega(f)=(f\otimes\iota)(W\Omega)$.

\begin{lem}\label{lambdaOmega}
The map $f\mapsto \lambda_\Omega(f)$ is a representation of $(A(G), \cdot_\Omega)$ on $B(L^2(G))$, i.e. 
$$\lambda_\Omega(f\cdot_\Omega g)=\lambda_\Omega(f)\lambda_\Omega(g) \text{ for all }f, g\in A(G).$$
Moreover, 
\begin{equation}\label{omegalambda}
\omega\lambda_\Omega(f)=M_{\check{\omega f}}\omega, f\in A(G).
\end{equation}
\end{lem}

 \begin{proof}
Let  $f$, $g\in A(G)$ and $\xi$, $\eta\in L^2(G)$. Write $\psi_{\xi,\eta}$ for the vector functional given by $\psi_{\xi,\eta}(T)=\langle T\xi,\eta\rangle$, $T\in B(L^2(G))$. Then 
\begin{eqnarray*}
\langle\lambda_\Omega(f\cdot_\Omega g)\xi,\eta\rangle&=&(((f\cdot_\Omega g)\otimes\iota)(W\Omega),\psi_{\xi,\eta})=((\Gamma\otimes\iota)(W\Omega), \Omega(f\otimes g)\otimes\psi_{\xi,\eta})\\&=&((\Gamma\otimes\iota)(W)(\Gamma\otimes\iota)(\Omega)(\Omega\otimes I), f\otimes g\otimes\psi_{\xi,\eta})\\
&\stackrel{(\ref{2co})}{=}&((\Gamma\otimes\iota)(W)(\iota\otimes \Gamma)(\Omega)(I\otimes \Omega), f\otimes g\otimes\psi_{\xi,\eta})\\
&=&(W_{12}^*W_{23}W_{12}(\iota\otimes \Gamma)(\Omega)(I\otimes \Omega), f\otimes g\otimes\psi_{\xi,\eta})\\
&\stackrel{(\ref{pentagon})}{=}&(W_{13}W_{23}(\iota\otimes \Gamma)(\Omega)(I\otimes \Omega), f\otimes g\otimes\psi_{\xi,\eta})\\
&=&(W_{13}W_{23}W_{23}^*\Omega_{13}W_{23}\Omega_{23},  f\otimes g\otimes\psi_{\xi,\eta})\\
&=&(W_{13}\Omega_{13}W_{23}\Omega_{23}, f\otimes g\otimes\psi_{\xi,\eta})\\&=&\langle (f\otimes\iota)(W\Omega)(g\otimes\iota)(W\Omega),\xi,\eta\rangle
=\langle\lambda_\Omega(f)\lambda_\Omega(g)\xi,\eta\rangle,
\end{eqnarray*} 
where the first equality in the  last line can be seen on elementary tensors and using then linearity and density arguments. In fact, if $X=a\otimes b$, $Y=c\otimes d$, $f(s)=\langle \lambda(s)\xi_1,\eta_1\rangle$ and $g(s)=\langle\lambda(s)\xi_2,\eta_2\rangle$, then
\begin{eqnarray*}
&&(X_{13}Y_{23}, f\otimes g\otimes\psi_{\xi,\eta})=\langle (a\otimes I\otimes b)(I\otimes c\otimes d)\xi_1\otimes\xi_2\otimes\xi,\eta_1\otimes\eta_2\otimes\eta\rangle\\&&=\langle a\xi_1,\eta_1\rangle\langle c\xi_2,\eta_2\rangle\langle bd\xi,\eta\rangle=\langle (f\otimes\iota) (X)(g\otimes\iota)(Y)\xi,\eta\rangle.
\end{eqnarray*}
The formula (\ref{omegalambda}) follows from the following calculations
\begin{eqnarray*}
\omega\lambda_{\Omega}(f)&=&(f\otimes\iota)((I\otimes\omega)W\Omega)=(f\otimes\iota)(W\Gamma(\omega)\Omega)\\&=& (f\otimes\iota)(W\omega\otimes\omega)=(\omega f\otimes\iota)(W)\omega=M_{\check{\omega f}}\omega.
\end{eqnarray*}
 \end{proof}

\section{{\sc \textbf{Complexification of $ G$}}}
\subsection{{\sc \textbf{Preliminaries on Unbounded Operators}}}

We start with some basic material on unbounded operators which will be used in the paper. Our main reference is \cite{schmudgen}. 

Let $\h$ be a Hilbert space with the inner product $\langle\cdot,\cdot\rangle$.  Recall that a linear operator $T$ defined on a subspace $\mathcal {D}(T)\subset \h$, called a domain of $T$, is said to be closed if the graph of $T$, $\{(\xi, T\xi)\;|\; \xi\in\h \}$, is closed in $\h\oplus\h$. Given linear operators $T$ and  $S$, we write $T\subset S$ if $\mathcal D(T)\subset\mathcal D(S)$ and $S|_{\mathcal D(T)}=T$; we say that $S$ is an extension of $T$. We have $T=S$ if $T\subset S$ and $S\subset T$. A linear operator $T$ is called closable if  it has a closed extension. Clearly, $T$ is closable if and only if the conditions $(\xi_n)_n\in \mathcal D(T)$, $\eta\in\h$, $\|\xi_n\|\to 0$ and $\|T\xi_n-\eta\|\to 0$ imply $\eta=0$. The minimal closed extension of a closable $T$ exists and will be denoted by $\bar T$.  We say that a subspace $\mathcal{U}\subseteq \mathcal{D}(T)$ is a \textit{core} for $T$ if for any $\xi\in \mathcal{D}(T)$, there is a sequence $(\xi_{n})_n\subset \mathcal{U},$  such that $\xi_{n}\to \xi$ and $T \xi_{n}\to T \xi.$ Equivalently, the subspace $\{(\xi,\;T\xi)\;|\;\xi\in \mathcal{U}\}\subseteq \h\oplus \h$ is dense in the graph of $T.$

If $T$ is an operator with a dense domain it has a well-defined adjoint  operator $T^*:\mathcal D(T^*)\to\h$, which is always a closed operator. An operator $T$ is  called selfadjoint if $T=T^*$; a selfadjoint operator is positive if it has a nonnegative spectrum. $T$ is essentially selfadjoint if $\bar T$ is selfadjoint. 

Any selfadjoint $T$ has a spectral measure $E_T$ on the $\sigma$- algebra $\mathcal B(\mathbb R)$ of Borel subsets of $\mathbb R$, and 
$$ T=\int_{\spec T} tdE_T(t); $$
if  $f$ is a Borel measurable  function, we write $f(T)$ for the operator
$$
\begin{array}{ccc}
f(T)={\displaystyle \int_{\spec T}} f(t)\;dE_T(t),& \mathcal D(f(A))=\{\xi\in H\;|\; {\displaystyle \int_{\spec T} }|f(t)|^2 \;d(E(t)\xi,\xi)<\infty\}.
\end{array}
$$

If $T$ is a closed operator with dense domain, then $T^*T$ is positive and $T$ has the polar decomposition $T=U|T|$, where $|T|=(T^*T)^{1/2}$ and $U$ is a partial isometry; $|T|$, $T$ and $U$ have the identical initial projections. 
\\

  We say that a closed operator $T$ defined on a dense domain $\mathcal{D}(T)\subseteq \h$ is {\it affiliated} with a von Neumann algebra $\mathcal M$ of $ B(\h)$ if $UT\subset TU$ for any unitary operator $U\in\mathcal M'$, where $\mathcal M'$ as usually stands for the commutant of $\mathcal M$.  Note that if $T=U|T|$ is the polar decomposition of $T$ then $T$ is affiliated with $\mathcal M$ if and only if $U\in\mathcal M$ and $|T|$ is affiliated with $\mathcal M$, the latter is equivalent that the spectral projections $E_{|T|}(\Delta)$, $\Delta\in\mathcal B(\mathbb R)$,  of $|T|$ belong to $VN(G)$.
We denote by $\overline {VN(G)}$ the set of all affiliated with $VN(G)$ elements.  We write $\overline{VN(G)}^+$ for the set of positive operators in $\overline{VN(G)}$.
If $T=T^*$ is affiliated with $VN(G)$, i.e. $E_T (\Delta)\in VN(G)$ for any $\Delta\in \mathcal B(\mathbb R)$,  then $f(T)\in\overline{VN(G)}$ for any Borel function $f$ on $\mathbb R$. 
\\

Let $A$ be a linear operator on $\h$. A vector $\varphi$ in $\h$ is called analytic  for $A$ if $\varphi\in\mathcal D(A^n)$ for all $n\in\mathbb N$ and if there exists a constant $M$ (depending on $\varphi$) such that
$$
\begin{array}{ccc}
\|A^n\varphi\|\leq M^n n!&\text{ for all }n\in \mathbb N.
\end{array}
$$
We write $\mathcal D_\omega(A)$ for the set of all analytic vectors of $A$.
If $A$ is selfadjoint with $E_A(\cdot)$ being the spectral measure of $A$, then $E_A(\Delta)\varphi$ is analytic for $A$ for any $\varphi\in \h $ and any bounded $\Delta\in\mathfrak B(\mathbb R)$, as
$$\|A^n E_A(\Delta)\varphi\|\leq M^n\|\varphi\|,$$
if $\Delta\subset[-M,M]$.
\\

It is known (see e.g. \cite[\proposition ~10.3.4]{schmudgen}) that if $T$ is a symmetric operator, i.e. $T\subset T^*$,  with a dense set of  analytic vectors, then $T$ is essentially selfadjoint.
\\

If $\mathcal U$ and $\mathcal V$ are subspace of $\h$ we write $\mathcal U\odot\mathcal V$ for the algebraic tensor product of the subspaces; $\h_1\otimes\h_2$ is the usual Hillbertian tensor product of two Hilbert spaces $\h_1$ and $\h_2$.

If $T_1$, $T_2$  are closed densely defined operator with the domains $\mathcal D(T_1)$ and $\mathcal D(T_2)\subset \h$ respectively, then the operator $T_1\otimes T_2$ with domain $\mathcal D(T_1)\odot \mathcal D(T_2)$ is closable. In fact, if $\xi_n\to 0$, where $\xi_n\in \mathcal D(T_1)\odot \mathcal D(T_2)$,  and $(T_1\otimes T_2)\xi_n\to\eta$ then for any $f\in \mathcal D(T_1^*)\odot \mathcal D(T_2^*)$, we have $\langle(T_1\otimes T_2)\xi_n,f\rangle=\langle\xi_n,(T_1^*\otimes T_2^*)f\rangle\to 0$, giving $\langle\eta,f\rangle=0$. As each $T_i$ is closed, $\mathcal D(T_i^*) $ is dense in $\h$ and hence $\mathcal D(T_1^*)\odot\mathcal D(T_2^*)$ is dense in $\h\otimes \h$, showing that $\eta=0$ and that $T_1\otimes T_2$ is closable. Unless otherwise stated we will write $T_1\otimes T_2$ for the corresponding closure.
\\

We say that two selfadjoint operators $T_1$, $T_2$ strongly commute, if
$$
\begin{array}{cccc}
E_{T_1}(\Delta_1)E_{T_2}(\Delta_2)= E_{T_2}(\Delta_2)E_{T_1}(\Delta_1), & \text{for all $\Delta_1,\Delta_2\in\mathcal B(\mathbb R),$}
\end{array}
$$
where $E_{T_i}(\cdot)$ is the spectral measure of $T_i$.
We define a product spectral measure $E_{T_1}\times E_{T_2}: \mathcal B(\mathbb R^2)\to \mathcal B(\h)$ by letting
$E_{T_1}\times E_{T_2}(\Delta_1\times\Delta_2)=E_{T_1}(\Delta_1)E_{T_2}(\Delta_2)$ for Borel measurable rectangle $\Delta_1\times\Delta_2$.
If $f:\mathbb R^2\to  \mathbb C$  is Borel measurable we set
$$f(T_1,T_2)=\int_{\mathbb R^2}f(x_1,x_2)\; dE_{T_1}\times E_{T_2}(x_1,x_2).$$ It is a selfadjoint operator if $f$ is real-valued.
\\

Let $T_i$ be selfadjoint operators, $i=1,2$. Then $T_1\otimes 1$ and $1\otimes T_2$ are selfadjoint operators that commute strongly. Then  $T_1\otimes T_2$ is  selfadjoint and 
$$T_1\otimes T_2=f(T_1\otimes 1, 1\otimes T_2),$$
where $f(x_1,x_2)=x_1x_2$. 
Observe that $T_1\otimes T_2$ is essentially selfadjoint on $\mathcal D(T_1)\odot\mathcal D(T_2)$, as $\mathcal D_\omega(T_1)\odot\mathcal D_{\omega}(T_2)$ is a dense subset of $\h_1\otimes\h_2$ and consists of analytic vectors for $T_1\otimes T_2$.

For closed densely defined operators $S_1$, $S_2$ with polar decomposition
$S_i=U_i|S_i|$, $i=1,2$, we have $S_1\otimes S_2=(U_1\otimes U_2)(|S_1|\otimes|S_2|)$ is the polar decomposition of the closed operator $S_1\otimes S_2$.


\subsection{{\sc \textbf{$\lambda$-Complexification of a Locally Compact Group}}}

Let $G$ be a locally compact group and  let $W$ be the fundamental multiplicative unitary on $L^2(G\times G)$ implementing the coproduct $\Gamma$ on $VN(G)$. We can extend $\Gamma$ to $\overline{VN(G)}$ by defining 
$$
\begin{array}{cccc}
\Gamma(T)=W^*(1\otimes T)W,&\text{for $T\in \overline{VN(G)}.$}
\end{array}
$$ Clearly, the unbounded operator $\Gamma(T)$ is closed.
If $T^*=T$ and $E_T(\cdot)$ is the spectral measure of $T$, then both operators $1\otimes T$ and $\Gamma(T)$ are selfadjoint with $1\otimes E_T(\cdot)$ and $\Gamma\circ E_T(\cdot)$ being the corresponding spectral measures. In particular,
$$\Gamma(T)=\int_{\mathbb R} t\; d(\Gamma\circ E_T(t)).$$
If $T=U|T|$ is the polar decomposition of $T$, $\Gamma(T)=\Gamma(U)\Gamma(|T|)$ is the polar decomposition of $\Gamma(T)$.

\begin{defn}
By the $\lambda$-complexification $ G_{\mathbb C,\lambda}$ of $G$ we shall mean the set of all   non-zero (unbounded) operators $T\in\overline{VN(G)} $ such that
\begin{equation}\label{defrel}
  \Gamma(T)=T\otimes T.
\end{equation}
\end{defn}
We note that $ G_{\mathbb C,\lambda}\cap VN(G)=\{T\in VN(G)|\; \Gamma(T)=T\otimes T, T\ne 0\}=\lambda(G)$, see e.g. \cite[Chapter 11, \theorem~16]{takesaki} giving an embedding of $G$ into $G_{\mathbb C,\lambda}$. 

Let
\begin{equation}\label{liealgebra}
    \Lambda=\{\alpha\in\overline{VN(G)}\;|\; \alpha^*=-\alpha,\; \alpha\otimes 1+1\otimes\alpha=\Gamma(\alpha)\}.
\end{equation}
As for $\alpha\in \Lambda$, the operators $i\alpha\otimes 1$ and $1\otimes i\alpha$ are selfadjoint and strongly commute, the sum
$i\alpha\otimes 1+1\otimes i\alpha$,
defined via the functional calculus, gives a selfadjoint operator; in (\ref{liealgebra})
 we require it to be equal to the selfadjoint operator $\Gamma(i\alpha)$.
\\

If $\alpha\in\Lambda$, we will define $\exp{z\alpha}$, $z\in\mathbb C$, through functional calculus, i.e.
$$\exp{z\alpha}=\exp{(-iz(i\alpha))}=\int_{\mathbb R}\exp{(-izt)}\; d E_{i\alpha}(t).$$
\begin{prop}
For $\alpha\in\Lambda$ and $z\in\mathbb C$, $\exp{z\alpha}\in G_{\mathbb C,\lambda}$.
\end{prop}

\begin{proof}
It follows from the functional calculus and definition of $\Gamma$ that
\begin{eqnarray*}
\Gamma(\exp{z\alpha})&=&W^*(1\otimes \exp{z\alpha})W=W^*\int_{\mathbb R}\exp{(-izt)}\;dE_{1\otimes i\alpha}(t)W\\
&=& \int_{\mathbb R} \exp{(-izt)}\; dE_{W^*(1\otimes i\alpha)W}(t)=\exp{z\Gamma(\alpha)}\\
&=&\exp{z(\alpha\otimes 1+1\otimes\alpha)}\\
&=&\exp{(z\alpha\otimes 1)}\exp{(1\otimes z\alpha)}=(\exp{z\alpha}\otimes 1)(1\otimes\exp{z\alpha})\\
&=&\exp{z\alpha}\otimes\exp{z\alpha}.
\end{eqnarray*}
\end{proof}

\begin{prop}\label{bijection}
The map $\alpha\in\Lambda\mapsto \exp{i\alpha}$ is a bijection onto  $G_{\mathbb C,\lambda}\cap\overline{VN(G)}^+$.
\end{prop}
\begin{proof}
That $\exp{i\alpha}$ is positive and affiliated with $VN(G)$ follows from the functional calculus.

Let $T\in G_{\mathbb C,\lambda}\cap\overline{VN(G)}^+$. Using arguments similar to those in the proof of the previous proposition we obtain
$$
\begin{array}{ccc}
\Gamma(T^{it})=W^*(1\otimes T^{it})W=\Gamma(T)^{it}=(T\otimes T)^{it}=T^{it}\otimes T^{it}, & \text{for $t\in\mathbb R.$}
\end{array}
$$
Let $A=\int_{\mathbb R^+}\ln t\;dE_T(t)$. Then $T^{it}=\exp{itA}$ and
$$
\begin{array}{cccc}
\exp{it\Gamma(A)}=\Gamma(\exp{itA})=\exp{itA}\otimes\exp{itA}=\exp{it(A\otimes 1+1\otimes A)}, &\text{for $t\in\mathbb R.$}
\end{array}
$$
By Stone's theorem about infinitisimal  generator of a strongly continuous unitary group, we obtain
$\Gamma(A)=A\otimes 1+1\otimes A$. Set $\alpha=-iA$. Then $T=\exp{i\alpha}$.
\end{proof}

\begin{prop}\label{complexi}
$$G_{\mathbb C,\lambda}=\{\lambda(s)\exp{i\alpha}\;|\; \alpha\in\Lambda,\; s\in G\}.$$
\end{prop}
\begin{proof}
As it was noticed before, if  $T=U|T|$ is the polar decomposition of $T\in G_{\mathbb C,\lambda}$ then $\Gamma(T)=\Gamma(U)\Gamma(|T|)$ and $T\otimes  T=(U\otimes U)|T|\otimes|T|$ are the polar decompositions of $\Gamma(T)$ and $T\otimes T$ respectively. Hence, by uniqueness of the polar decomposition, the equality $\Gamma(T)=T\otimes T$ implies $\Gamma(U)=U\otimes U$ and $\Gamma(|T|)=|T|\otimes|T|$. As $\lambda(G)=\{\lambda(s)\; |\; s\in G\}$ is precisely the family of non-zero bounded operators in $G_{\mathbb C,\lambda}$, it gives  $U=\lambda(s)$ for some $s\in G$. The statement now follows from \proposition~\ref{bijection}. \end{proof}

Let $G$ be a connected Lie group and $\mathcal g$ its Lie algebra with the exponential map $\exp_G:\mathcal g\to G$.  Let $\pi:G\to B(\h_\pi)$ be a unitary representation of $G$. A vector $\varphi\in \h_\pi$ is called a $C^\infty$-vectors for $\pi$ if the map $s\to\pi(s)\varphi$ from the $C^\infty$-manifold $G$ to $\h_\pi$ is a $C^\infty$-mapping. We write $\mathcal D^\infty(\pi)$ for the set of $C^\infty$-vectors for $\pi$. For $X\in\mathcal g$ we define the operator $d\pi(X)$ with domain $\mathcal D^\infty(\pi)$ by
$$
\begin{array}{cccc}
d\pi(X)\varphi=\frac{d}{dt}\pi(\exp_G{(tX))}\varphi|_{t=0},&\text{for $ \varphi\in \mathcal D^\infty(\pi).$}
\end{array}
$$
It is known that $id\pi(X)$ is essentially self-adjoint. We denote its self-adjoint closure by $i\partial\pi(X)$ which is the infinitesimal generator of the strongly continuous one-parameter unitary group $t\mapsto \pi(\exp_G(tX))$, i.e.
$$\pi(\exp_G{(tX)})=\exp{(t\partial\pi(X))}.$$

\begin{prop}\label{factorisation}
Let $G$ be a connected Lie group with Lie algebra $\mathcal g$. Then $\Lambda=\{\partial\lambda(X)\;|\; X\in\mathcal g\}$ and $G_{\mathbb C,\lambda}=\{\lambda(s)\exp{(i\partial\lambda(X))}\;|\; s\in G,\; X\in\mathcal g\}$.
\end{prop}

\begin{proof}
If $\alpha\in\Lambda$ then $\{\exp{(t\alpha)}|\; t\in\mathbb R\}$ is a strongly continuous one parameter group in $\lambda(G)\subset VN(G)$. Moreover,
$$
\begin{array}{ccc}
\langle \exp{(t\alpha)},\bar\eta\ast\check\xi\rangle=\langle \exp{(t\alpha)}\xi,\eta\rangle, &\text{for $\xi,\eta\in L^{2}(G),$}
\end{array}
$$
and $\{\exp{(t\alpha)}|\; t\in\mathbb R\}$ is continuous in the weak$^*$ topology on $VN(G)$ with the weak$^*$-limit
$w^*-\lim_{t\to 0}\exp{(t\alpha)}=1$.  Since $\lambda:G\to \lambda(G)\subset VN(G)$ is a homeomorphism when $VN(G)$ carries weak$^*$ topology it follows that 
$\lambda^{-1}(\exp{t\alpha})$ is a continuous one-parameter subgroup of $G$. Therefore there exists $X\in \mathcal g$ such that $\lambda^{-1}(\exp{(t\alpha)})=\exp{(tX)}$ and $\lambda(\exp{(tX)})=\exp{(t\alpha)}$, $t\in\mathbb R$,  giving $\partial \lambda(X)=\alpha$ and $$\Lambda\subset \{\partial\lambda(X)\;|\; X\in\mathcal g\}.$$

To see the reverse inclusion, we note that  $\partial\lambda(X)\in \overline{VN(G)}$, $\Gamma(\exp{(t\partial\lambda(X))})=\exp{(t\Gamma(\partial\lambda(X)))}$ and
$$
\begin{array}{ccc}
\exp{(t\Gamma(\partial\lambda(X)))}=\exp{(t\partial\lambda(X))}\otimes \exp{(t\partial\lambda(X))}, &\text{for $ t\in\mathbb R.$}
\end{array}
$$ Since
$\lim_{t\to 0}t^{-1}[\exp{(tV)}\varphi-\varphi]=V\varphi$ for any closed skew adjoint operator $V$ and $\varphi\in \mathcal D(V)$, we can easily obtain that $\partial\lambda(X)\in\Lambda$.

\end{proof}

\begin{rem*}\rm
Our definition is motivated by the work of McKennon \cite{McK1} and Cartwright and McMullen \cite{cartwright_mcmullen}, where they developed an abstract Lie theory for general, not necessarily Lie, compact groups. If we choose representatives $\pi_j:G\to B(H_j)$ of the isomorphism classes of irreducible (finite-dimensional) unitary representations of $G$ and identify $VN(G)$ with the $\ell^\infty$-sum of $B(H_j)$, then $\overline{VN(G)}=\prod_jB(H_j)\simeq \text{Trig}(G)^\dagger$, where $ \text{Trig}(G)^\dagger$  is the linear dual of the span of coefficients of irreducible unitary representations of $G$. In this case we have that $G_{\mathbb C,\lambda}$ coincides with the complexification $G_{\mathbb C}$  from \cite{cartwright_mcmullen, McK1} . We have that $G_{\mathbb C}$ is a group and as $G\simeq \{X\in VN(G)\;|\;  \Gamma(X)=X\otimes X, X\ne 0\}$ (\cite{takesaki}), one has a Cartan decomposition $G_{\mathbb C}\simeq G\cdot\exp{i\Lambda}$. 
If $G$ is a compact  connected Lie group then $G_{\mathbb C}=G_{\mathbb C,\lambda}$  coincides with the well-known construction of the universal complexification of $G$ due to Chevalley and the Lie algebra of $G_{\mathbb C}$ is the complexification $\mathcal  g_{\mathbb C}$ of the Lie algebra $\mathcal g\simeq\Lambda$ of $G$, where the usual Lie bracket $[X,Y]=XY-YX$ is considered in $\Lambda$. For instance $\mathbb T_{\mathbb C}\simeq \mathbb C^*$ and $(SU(n))_{\mathbb C}\simeq SL(n,\mathbb C)$. 

The concept of complexification was later generalised from compact to general locally compact groups in \cite{McK2} by McKennon, where the group $W^*$-algebra $W^*(G)$ was used instead of $VN(G)$. Our construction is an adaptation of McKennon's idea to the group von Neumann algebra setting. We have chosen this approach as it fits better our purpose to describe the spectrum of Beurling-Fourier algebras. 
As for the compact  group case McKennon's complexification $G_{\mathbb C}$ admits a factorisation $G_{\mathbb C}=G_\gamma\cdot G_{\mathbb C}^+$, where $G_\gamma$ is the image of $G$ under the canonical monomorphism $\gamma$ from $G$ to the group of unitary elements of $W^*(G)$ (compare this to the factorisation in \proposition~\ref{factorisation}).     However unlike the compact case, the unboundedness of elements in $G_{\mathbb C}^+$ and also $G_{\mathbb C,\lambda}\cap \overline{VN(G)}^+$ causes a problem in considering $G_{\mathbb C}$ and $G_{\mathbb C,\lambda}$ as groups, see \cite[section 4]{McK2}. A relation to the universal complexification of $G$, when $G$ is a  Lie group, is also unclear in general. However, in many interesting examples considered in \cite{gllst} we have $G_{\mathbb C,\lambda}= \lambda_{\mathbb C}(G_{\mathbb C}^u)$, where $G_\mathbb C^u$ is the universal complexification of $G$ and $\lambda_{\mathbb C}$ is the extension of the left regular representation to $G_{\mathbb C}^u$; the equality means  that for any $\varphi\in G_{\mathbb C,\lambda}$ there exists $g\in G_{\mathbb C}^u$ such that such that $\varphi=\overline{\lambda_{\mathbb C}(g)}$, see the discussion in \cite[section 2.3]{gllst}; in those cases one also has the Cartan decomposition
$$G_{\mathbb C}^u\simeq G\cdot\exp_{\mathbb C}(i\mathcal g),$$
where $\exp_{\mathbb C}$ is the extension of the exponential map to the complexification $\mathcal g_{\mathbb C}$ of the Lie algebra $\mathcal g$ of $G$.
It seems an interesting question to investigate the group structure of $G_{\mathbb C,\lambda}$ but it diverges from the main purpose of this paper. 

\end{rem*}

\section{\sc{\textbf{The spectrum of  Beurling-Fourier algebra and complexification}}}

In this section we establish sufficient conditions in terms of groups and weight inverses for the inclusion of the Gelfand spectrum of $A(G,\omega)$ into the $\lambda$-complexification $G_{\mathbb C,\lambda}$ of $G$, generalising some earlier results from \cite{lst} and \cite{gllst}.

\subsection{Point-spectrum Correspondence}
Let $\phi:A(G,\omega)\to  \mathbb{C}$ be a character of $A(G,\omega)$. By the duality~\eqref{duality}, there is a unique $\sigma\in VN(G)$ such that for any $\omega u\in A(G,\omega)$ we have $\phi(\omega u)=(\sigma, \omega u)_\omega=(\sigma,u)$.  The multiplicativity of $\phi$ gives
\begin{equation}\label{spec}
\sigma\otimes \sigma=\Gamma(\sigma)\Omega.
\end{equation}
and moreover, every $\sigma\in VN(G)$ satisfying~\eqref{spec} gives rise to a unique point in the spectrum $\spec\, A(G,\omega).$
In fact, for $u$, $v\in A(G)$, 
\begin{eqnarray*}
\phi((\omega u)(\omega v))&=&(\sigma, (\omega u)(\omega v))_\omega=(\sigma, \omega\Gamma_*(\Omega(u\otimes v)))_\omega\\&=&(\sigma, \Gamma_*(\Omega(u\otimes v)))=(\Gamma(\sigma)\Omega,u\otimes v);
\end{eqnarray*}
on the other hand
\begin{eqnarray*}
\phi(\omega u)\phi(\omega v)&=&(\sigma, \omega u)_\omega(\sigma, \omega v)_\omega=(\sigma, u)(\sigma, v)=(\sigma\otimes\sigma,u\otimes v),
\end{eqnarray*}
giving (\ref{spec}).

We can thus identify $\spec \,A(G,\omega)$ as the set of all non-zero elements $\sigma\in VN(G)$ satisfying~\eqref{spec}, i.e.
$$\spec\, A(G,\omega)=\{\sigma\in VN(G)\; |\;  \Gamma(\sigma)\Omega=\sigma\otimes\sigma, \sigma\ne 0\}.$$
Note that $\spec\, A(G,\omega)$ depend on the $2$-cocycle $\Omega$ rather than the weight inverse $\omega.$ Moreover, by~\eqref{spec}, for any $\sigma\in A(G,\omega)$
$$
\sigma\sigma^{*}\otimes\sigma\sigma^{*}=\Gamma(\sigma)\Omega\Omega^{*}\Gamma(\sigma)^{*}\leq \Gamma(\sigma\sigma^{*}),
$$
thus satisfying condition~\eqref{iw} in the definition of a weight inverse. It is a question whether $\sigma$ also satisfies~\eqref{kernel}. We will see in this section that in many cases (though we conjecture all) the elements in $\spec\, A(G,\omega)$ are again weight inverses.
\\

We let $S$ be the antipode of $VN(G)$;  this is an anti-isomorphism of $VN(G)$  given by $S(\lambda(s))=\lambda(s^{-1})$, $s\in G$. If $W$ is the multiplicative unitary and $w\in B(L^2(G))_*$, then 
\begin{equation}\label{anti}
S((\iota\otimes w)(W^*))=(\iota\otimes w)(W).
\end{equation}   
We refer to \cite{enock-shwartz} for background on the theory of Hopf-von-Neumann algebras but warn that our notations may differ from those in \cite{enock-shwartz}. 

Throughout the rest of this section, we use $\h$ for $L^{2}(G)$ and write 
$\psi_{\xi,\eta}$ to denote the normal functional  on $B(\h)$ given by $\psi_{\xi,\eta}(x)=\langle x\xi,\eta\rangle$,  $x\in B(\h)$.

\begin{lem}\label{star_lemma}
Let $\sigma\in \spec\, A(G,\omega)$. Then 
$$\psi_{\xi,\tilde\xi}(S(\iota\otimes\psi_{\sigma^*\eta,\tilde\eta}(\Omega^*W^*)))=
\langle(1\otimes\sigma^*)W(S(\sigma^*)\otimes 1)\xi\otimes\eta,\tilde\xi\otimes\tilde\eta\rangle,$$
for any $\xi$, $\eta$, $\tilde\xi$, $\tilde\eta\in \h$.
\end{lem}
\begin{proof}
From (\ref{anti}) and $\Omega^*\Gamma(\sigma^*)=\sigma^*\otimes\sigma^*$,  we have
\begin{eqnarray*}
&&\psi_{\xi,\tilde\xi}(S(\iota\otimes\psi_{\sigma^*\eta,\tilde\eta}(\Omega^*W^*)))
=\psi_{\xi,\tilde\xi}(S(\iota\otimes\psi_{\eta,\tilde\eta}(\Omega^*W^*(1\otimes\sigma^*)))\\
&&=\psi_{\xi,\tilde\xi}(S(\iota\otimes\psi_{\eta,\tilde\eta}(\Omega^*\Gamma(\sigma^*)W^*)))
=\psi_{\xi,\tilde\xi}(S(\iota\otimes\psi_{\eta,\tilde\eta}(\sigma^*\otimes\sigma^*W^*)))\\
&&=\psi_{\xi,\tilde\xi}(S(\sigma^*(\iota\otimes\psi_{\eta,\tilde\eta}((1\otimes\sigma^*)W^*)))
=\psi_{\xi,\tilde\xi}(S(\sigma^*(\iota\otimes\psi_{\eta,\sigma\tilde\eta}(W^*))\\
&&=\psi_{\xi,\tilde\xi}(S(\iota\otimes\psi_{\eta,\sigma\tilde\eta}(W^*))S(\sigma^*))
=\psi_{\xi,\tilde\xi}((\iota\otimes\psi_{\eta,\sigma\tilde\eta}(W))S(\sigma^*))\\
&&=\psi_{\xi,\tilde\xi}\otimes\psi_{\eta,\sigma\tilde\eta}(W(S(\sigma^*)\otimes 1))
=\psi_{\xi,\tilde\xi}\otimes\psi_{\eta,\tilde\eta}((1\otimes\sigma^*)W(S(\sigma^*)\otimes 1))\\
&&=\langle (1\otimes\sigma^*)W(S(\sigma^*)\otimes 1)\xi\otimes\eta,\tilde\xi\otimes\tilde\eta\rangle.
\end{eqnarray*}
\end{proof}
\begin{prop}\label{thm1}
Let $G$ be a locally compact group and let $\sigma\in \spec\, A(G,\omega)$. Assume that  $\sigma^{*}(\h)\cap\omega^{*}(\h)\ne\{0\}$. Then
$$S(\omega)\omega=S(\sigma)\sigma.$$
\end{prop}

\begin{rem*} \rm
It has been known for compact groups (\cite{lst}) and some Lie groups with certain weights  (\cite{gllst}) that the operators $\sigma\omega^{-1}$, $\sigma\in \spec\, A(G,\omega)$, are "points" of the complexification $G_{\mathbb C}$.  From this, the claim of the proposition becomes  intuitively quite clear. Formally, if there is an element $T\in G_{\mathbb{C}}$ such that $\sigma=T\omega$ then, as $S(T)=T^{-1}$ (the antipode "inverts" the elements of $G$ and $G_{\mathbb C}$),  we would have $S(\sigma)\sigma=S(T\omega)T\omega=S(\omega)T^{-1}T\omega=S(\omega)\omega.$
\end{rem*}
\begin{proof}

Let $\eta$ and $\zeta$ in $\h$ be such that $\sigma^*\zeta=\omega^*\eta\ne 0$. By  \lemma~\ref{star_lemma}, we have
$$(1\otimes\omega^*)W(S(\omega^*)\otimes 1)\xi\otimes\zeta= (1\otimes\sigma^*)W(S(\sigma^*)\otimes 1)\xi\otimes\eta$$
for any $\xi\in \h$.

Multiplying both hand sides of the equality from the left by $\Omega^*W^*$  and using the equality  $\Omega^*\Gamma(\sigma)^*=\sigma^*\otimes\sigma^*$ which holds  for all $\sigma\in \spec\, A(G,\omega)$ and in particular for $\omega$,
we conclude that
$$
\begin{array}{cccc}
\omega^*S(\omega^*)\xi\otimes\omega^*\eta=\sigma^*S(\sigma^*)\xi\otimes\sigma^*\zeta, &\text{for all $\xi \in  \h,$}
\end{array}
$$ and hence
$\omega^*S(\omega)^*=\sigma^*S(\sigma)^*.$
\end{proof}

\begin{rem*}\rm
The following formal calculations support
the idea that the above proposition might be true for any $\sigma\in \spec\, A(G,\omega)$.

Consider $M=(S\otimes\iota)(W\Omega)W\Omega.$ From~\eqref{weight} it follows that
$(I\otimes\omega)W\Omega=W(\omega\otimes\omega)$,
and hence $(I\otimes \omega)(S\otimes\iota)(W\Omega)=(S(\omega)\otimes I)W^{*}(I\otimes \omega).$ We then calculate
\begin{eqnarray*}
(I\otimes \omega)M&=&(S(\omega)\otimes I)W^{*}(I\otimes \omega)W\Omega= (S(\omega)\otimes I)W^{*}W(\omega\otimes \omega)\\&=&(S(\omega)\omega)\otimes \omega=
(I\otimes \omega)(S(\omega)\omega\otimes I).
\end{eqnarray*}
As $\ker\omega=\{0\}$ we get $M=(S(\omega)\omega)\otimes I.$ Let now $\sigma \in \spec\, A(G,\omega)$ be arbitrary. Similarly, 
$$
(I \otimes\sigma)M=(S(\sigma)\sigma)\otimes \sigma
$$
and hence $(S(\sigma)\sigma)\otimes \sigma=(S(\omega)\omega)\otimes \sigma.$ Therefore, 
$S(\sigma)\sigma=S(\omega)\omega.$

The calculations are only formal as $S$ is not  a completely bounded map in general and hence $S\otimes\iota$ is not defined on the whole $VN(G)\bar\otimes B(\h)$. By \cite[\proposition ~1.5]{forrest-runde}, $S$ is completely bounded if and only if $G$ is virtually abelian, i.e. has an abelian subgroup of finite index.  Consequently, for such $G$, $S(\sigma)\sigma=S(\omega)\omega$ for any $\sigma\in \spec\, A(G,\omega)$.
 
\end{rem*}
\begin{cor}\label{trivialkernel}
Let $\sigma\in \spec\, A(G,\omega)$. If $\sigma^{*}(\h)\cap\omega^{*}(\h)\ne\{0\}$ then $\ker \sigma=\{0\}.$
\end{cor}
\begin{proof}
This follows from \proposition ~\ref{thm1}, as $\ker\sigma\subseteq \ker S(\sigma)\sigma=\ker S(\omega)\omega=\{0\}.$
\end{proof}
A  natural question is when $\sigma\in \spec\, A(G,\omega)$ is again a weight inverse. Clearly,
$$
\sigma\sigma^{*}\otimes\sigma\sigma^{*}=\Gamma(\sigma)\Omega\Omega^{*}\Gamma(\sigma^{*})\leq \Gamma(\sigma\sigma^{*})
$$
and hence the first condition~\eqref{iw} of being a weight inverse is fulfilled. 
The same arguments as in \corollary~\ref{trivialkernel} show that if $S(\omega)\omega=S(\sigma)\sigma$ then $\ker \sigma=\{0\}.$ An issue is to obtain $\ker\sigma^{*}=\{0\}$.  We will adopt the extra condition $\ker\; \Omega^{*}=\{0\}$ as a way to guarantee it.
\begin{lem}\label{trivcoker}
If $\ker\; \Omega^{*}=\{0\},$ then $\ker\sigma^{*}=\{0\}$ for every $\sigma\in \spec\, A(G,\omega).$
\end{lem}
\begin{proof}
 Let $\sigma\in \spec\, A(G,\omega).$ As $\ker\; \Omega^{*}=0,$ we have
$$
\ker\; \sigma^{*}\otimes\sigma^{*}=\ker\;\Omega^{*}\Gamma(\sigma^{*})=\ker\; \Gamma(\sigma^{*}).
$$
Thus, if we let $P$  denote the projection onto $\ker \sigma^{*},$ then $P\in VN(G)$ and $P$ satisfies
\begin{equation}\label{P}
P\otimes P=\Gamma(P).
\end{equation} 
By~\cite[Chapter 11, \theorem~16]{takesaki},  either $P=0$ or $P=\lambda(s)$ for some $s\in G$; the latter is possible only if $P=\lambda(e)=I$ and hence $\sigma=0$ contradicting $\sigma\in \spec\, A(G,\omega).$ Therefore $P=0$, i.e. $\ker\sigma^*=\{0\}$.
\end{proof}

Here comes the main result of this section that establishes a connection between (a part of) the spectrum $ \spec\, A(G,\omega)$ and $G_{\mathbb C,\lambda}$. 

\begin{thm}\label{mainthm}
Let $\omega\in VN(G)$ be a weight inverse on the dual of $G$ and let $\sigma\in \spec\, A(G,\omega)$ be such that $\sigma^{*}(\h)\cap\omega^{*}(\h)\ne\{0\}$. Assume $\ker \;\Omega^{*}=\{0\}.$ Then
\begin{enumerate}[$(i)$]
\item $\sigma$ is a weight inverse,
\item $S(\sigma)\sigma=S(\omega)\omega,$ 
\item the linear operator
$$
\begin{array}{ccc}
T_\sigma: \omega\xi\mapsto\sigma\xi,& \text{for $\xi\in \h,$}
\end{array}
$$
is closable with the closure  in $G_{\mathbb{C},\lambda}.$
\end{enumerate}
\end{thm}

\begin{proof}
$(i)$   follows from \corollary~\ref{trivialkernel} and \lemma~\ref{trivcoker},  together with the earlier observation that $\sigma\sigma^{*}\otimes\sigma\sigma^{*}\leq \Gamma(\sigma\sigma^{*}).$ 

$(ii)$   follows from \proposition~\ref{thm1}.

$(iii)$
The operator $T_\sigma$ is well-defined as $\ker\omega=\{0\}$. Let $\xi_n\in \h$ be such that $\omega\xi_n\to 0$ and $\sigma\xi_n\to y$. Then for any $\xi\in \h$,
\begin{eqnarray*}
\langle y, S(\sigma^*)\xi\rangle=\lim_{n\to\infty}\langle S(\sigma)\sigma\xi_n,\xi\rangle=\lim_{n\to\infty}\langle S(\omega)\omega\xi_n,\xi\rangle=0.
\end{eqnarray*}
Therefore $y\perp \overline{S(\sigma^*)(\h)}$.  By \lemma ~\ref{trivcoker},  $\ker\sigma^*=\{0\}$.  This yields $\ker S(\sigma)=\{0\}$ (since if $P$ is the range projection for $A\in VN(G),$ then $S(P)$ is the range projection for $S(A^{*})$).  Therefore
$S(\sigma^*)\h$ is dense in $\h$ and hence $y=0$. Consequently,  $T_\sigma$ is closable.

Write $T_\sigma$ also for the closure. Then $T_{\sigma}$ is affiliated with $VN(G)$, and even more,  it is affiliated with the von Neumann algebra ${\mathcal N}(\omega,\sigma)$ generated by $\omega$ and $\sigma.$ In fact, let $V\in {\mathcal N}(\omega,\sigma)'$ be a unitary. Then for any $\xi\in \h$ of the form $\xi=\omega \eta,$ we have $$V T_{\sigma}\xi=V\sigma \eta=\sigma V\eta=T_{\sigma}\omega V\eta=T_{\sigma}V \xi$$
and hence $V^{*}T_{\sigma}V=T_{\sigma},$ showing the statement.

The only claim left to prove is that $\Gamma(T_\sigma)=T_\sigma\otimes T_\sigma$.
Observe first that
$$
(\omega\otimes \omega)(\h\otimes \h)=\Gamma(\omega) \Omega(\h\otimes \h)\subseteq
$$
$$
\subseteq \Gamma(\omega) (\h\otimes \h)\subseteq\mathcal D( \Gamma(T_{\sigma})),
$$
and
$$
\Gamma(T_{\sigma})\Gamma(\omega) \Omega(\h\odot\h)=\Gamma(T_{\sigma}\omega) \Omega(\h\odot\h)=\Gamma(\sigma) \Omega(\h\odot\h)=
$$
$$
=(\sigma\otimes \sigma)(\h\odot \h)=(T_{\sigma}\otimes T_{\sigma})(\omega\otimes \omega)(\h\odot \h).
$$
 We have
$$T_\sigma\otimes T_\sigma|_{\omega(\h)\odot\omega(\h)}=\Gamma(T_\sigma)|_{\omega(\h)\odot\omega(\h)}.$$
By convention, $T_\sigma\otimes T_\sigma$ is the closure of the operator $T_\sigma\odot T_\sigma$ defined on $\mathcal D(T_\sigma)\odot\mathcal D(T_\sigma)$ or, equivalently, on $\omega(\h)\odot \omega(\h)$, as $\omega(\h)$ is a core of $T_\sigma$. Hence
$$\Gamma(T_\sigma)\supset T_\sigma\otimes T_\sigma.$$ 
To see the equality, we must prove that  $\overline{\Gamma(T_\sigma)|_{\omega(\h)\odot\omega(\h)}}=\Gamma(T_\sigma)$. To do this we note first that $\Gamma(\omega)(\h\otimes \h)$ is a core for $\Gamma(T_\sigma)$ and hence the linear subspace
\begin{equation}\label{set}
 \{(x,\;\Gamma(T_\sigma)x)\;|\; x\in \Gamma(\omega)(\h\otimes\h)\}
\end{equation}
is dense in the graph of $\Gamma(T_\sigma)$.
Therefore, it is enough to see that  the closure
of
$$\{(x,\Gamma(T_\sigma)x)\;|\; x\in \Gamma(\omega)\Omega(\h\odot\h)\}=\{(x,\;\Gamma(T_\sigma)x)\;|\; x\in \omega\h\odot \omega\h)\}
$$
contains (\ref{set}).

As $\Omega(\h\odot \h)$ is dense in $\h\otimes\h$, we have that for any $\Gamma(\omega)\xi$, $\xi\in\h\otimes \h$, there exists $(\xi_n)_n\subset \h\odot \h$ such that $\Omega\xi_n\to \xi$ and hence $\Gamma(\omega)\Omega\xi_n\to\Gamma(\omega)\xi$. Moreover,
$$\Gamma(T_\sigma)\Gamma(\omega)\Omega\xi_n=\Gamma(\sigma)\Omega\xi_n\to\Gamma(\sigma)\xi=\Gamma(T_\sigma)\Gamma(\omega)\xi,$$
showing  the claim.
\end{proof} 
We remark that 
\begin{equation}\label{condsigma}
\sigma^{*}(\h)\cap\omega^{*}(\h)\ne\{0\}
\end{equation}
 for $\sigma\in \spec\, A(G,\omega)$ means that the domain $\mathcal D(T_\sigma^*)$ of the operator $T_\sigma^*=(\sigma\omega^{-1})^*=(\omega^*)^{-1}\sigma^*$ is not zero.  The theorem says that in this case $\mathcal D(T_\sigma^*)$ is large enough to be dense in $\h$, as the latter is equivalent to the closability of $T_\sigma$.

In what follows we shall use the notation $T_\sigma$ for the closed operator $\bar T_\sigma$ when there is no risk of confusion. 

We derive now a number of  consequences from the previous theorem. We assume that $\ker\,\Omega^*=\{0\}$.  

\begin{cor}\label{isomorph}
For $\sigma\in \spec\, A(G,\omega)$ as  in \theorem~\ref{mainthm}, there is a natural isometric isomorphism
$$A( G,\sigma)\cong A(G,\omega),$$
$$\sigma f\mapsto \omega f.$$
\end{cor}
\begin{proof}
This is immediate from the definitions of the norm and product on the corresponding  spaces:
$$
\begin{array}{ccc}
||\omega f||_{\omega}=||f||=||\sigma f||_{\sigma}, &\text{for $ f\in A(G),$}
\end{array}
$$
$$
\begin{array}{ccc}
(\omega g)(\omega h)=\omega \Gamma_{*}(\Omega(g\otimes h)), & (\sigma g)(\sigma h)=\sigma \Gamma_{*}(\Omega(g\otimes h)), & \text{for $g,h\in A(G).$}
\end{array}
$$
\end{proof}
We remark that the above corollary is also clear from the discussion after the proof of \proposition~\ \ref{2weights}. 

\begin{cor}\label{spectrum}
For $\sigma\in \spec\, A(G,\omega)$  we have  $\sigma^{*}(\h)\cap\omega^{*}(\h)\ne\{0\}$  if and only if 
$\sigma=T\omega$ for $T\in G_{\mathbb C,\lambda}$ such that $\omega(\h)\subset \mathcal D(T)$.
Consequently, if $\sigma^{*}(\h)\cap\omega^{*}(\h)\ne\{0\}$ for any $\sigma\in \spec\, A(G,\omega)$ then 
$$\spec\, A(G,\omega)\subset \{T\omega\;|\;T\in G_{\mathbb C,\lambda}, \omega(\h)\subset \mathcal D(T)\}.$$
\end{cor}
\begin{proof}
The "only if" part follows from \theorem ~\ref{mainthm}. If $\sigma=T\omega$ for $T\in G_{\mathbb C,\lambda}$ then $\sigma^*\supset \omega^*T^*$ giving the "if" part. 
\end{proof}

\begin{rem}\label{vnGomega} \rm
In \cite{gllst} the dual $A(G,\omega)^*$ is identified with the weighted space $VN(G,\omega)$ given by $$VN(G,\omega):=\{A\omega^{-1}\;|\; A\in VN(G)\}$$
with the norm
$\|A\omega^{-1}\|_{VN(G,\omega)}=\|A\|$ via
$$( A\omega^{-1},\omega u):=( A,u ).$$
Then the spectrum  of $A(G,\omega)$ is considered as a subset of $VN(G,\omega)$ instead of $VN(G)$. Clearly we have the isometry $\Phi: VN(G)\to VN(G,\omega)$, $A\mapsto A\omega^{-1}$. With this identification we have that if $\sigma^{*}(\h)\cap\omega^{*}(\h)\ne\{0\}$ for any $\sigma\in \spec\, A(G,\omega)$, then  
\begin{equation}\label{inclusion}
\spec\, A(G,\omega)\simeq\{T\in G_{\mathbb C,\lambda}|\;  T\omega\in VN(G)\}\subset G_{\mathbb C,\lambda}.
\end{equation}
\end{rem}
Next, we prove a 'partial converse' of \theorem~\ref{mainthm}, which shows that  every element in $G_{\mathbb{C},\lambda}$ can be seen  as coming from a weight inverse.
\begin{prop}\label{weight_comp}
If $T\in G_{\mathbb{C},\lambda}$ then there exists a weight inverse $\omega\in VN(G)$ and $\sigma\in \spec\, A(G,\omega)$ such that $T=T_{\sigma}.$
\end{prop}
\begin{proof} 
Let $T\in G_{\mathbb{C},\lambda}$ and $U |T|$ be its polar decomposition. Then $U=\lambda(s)$ for some $s\in G$ and $\Gamma(|T|)=|T|\otimes |T|. $ Hence $\Gamma(|T|^{i t})=|T|^{it}\otimes |T|^{it}$ and $ |T|^{it}$ determines  a strongly continuous representation $\psi: \mathbb{R}\to \lambda(G)\subseteq B(\h)$ by setting $\psi(t)=|T|^{it}.$ By the standard theory, the map
$$
\begin{array}{ccc}
\hat{f}(x)={\displaystyle \int_{\mathbb{R}}} f(t)e^{i x t}\;dt\;\mapsto\; {\displaystyle \int_{\mathbb{R}}}f(t) \psi(t)\; dt\in VN(G), &\text{for $ f\in L^{1}(\mathbb{R}),$}
\end{array}
$$
extends to a $*$-homomorphism $\varphi:C^{*}(\mathbb{R})\cong C_{0}(\mathbb{R})\to VN(G);$ we have $\varphi(f)=f(\ln(|T|))$. 
The image of $C_{0}(\mathbb{R})$ is clearly non-degenerate, and hence we can extend $\varphi$ in a unique way to a homomorphism $\overline{\varphi}:C_{b}(\mathbb{R})\to VN(G).$ If we let $\overline{\varphi\otimes\varphi}$ denote the extension of the map $\varphi\otimes\varphi:C_{0}(\mathbb{R})\otimes C_{0}(\mathbb{R})\to VN(G)\bar\otimes VN(G)$ to $C_{b}(\mathbb{R}\times\mathbb{R}),$ then it is easy to see from the uniqueness of the extensions that the diagram
\begin{equation}\label{comdia}
\begin{xy}\xymatrixcolsep{4pc}
\xymatrix{
C_{b}(\mathbb{R}\times \mathbb{R})\ar[r]^{\overline{\varphi\otimes\varphi}} & VN(G)\bar\otimes VN(G)\\
C_{b}(\mathbb{R}) \ar[r]^{\overline{\varphi}}\ar[u]^{\Gamma_{\mathbb{R}}}  &VN(G)\ar[u]^{\Gamma}
}\end{xy}
\end{equation}
is commutative; here we write $\Gamma_{\mathbb{R}}$ for the restriction of the coproduct to $C_{b}(\mathbb{R})$. 

Now if we let $\omega=\varphi(e^{-2|x|})$, then $\omega^{2}\otimes\omega^{2}\leq \Gamma(\omega^{2})$, and the non-degeneracy of $\varphi$ gives $\ker\;\omega=\{0\}.$ Thus $\omega$ is a weight inverse in $VN(G).$ Moreover, the $2$-cocycle associated to $\omega$ is given by \begin{equation}\label{line2co}\Omega=\overline{\varphi\otimes\varphi}(e^{2|x+y|-2|x|-2|y|}).\end{equation}
If we let $\sigma=\lambda(s)\varphi(e^{x-2|x|}),$ then it is easy to see from~\eqref{line2co} that $\Gamma(\sigma)\Omega=\sigma\otimes\sigma$ and hence $\sigma\in \spec\, A(G,\omega).$ Moreover, the closure of the unbounded operator 
$$
\begin{array}{ccc}
\omega \xi=e^{-2|\ln |T||}\xi\mapsto \sigma \xi=\lambda(s)e^{\ln |T|-2|\ln |T||}\xi, & \text{for $\xi\in\h,$}
\end{array}
$$
is given by $T.$ 
\end{proof}

Next we derive some further properties of $\spec\, A(G,\omega)\cap G_{\mathbb C,\lambda}\omega$. 

\begin{lem}\label{polardec}
Let $\sigma\in \spec\, A(G,\omega)\cap G_{\mathbb C,\lambda}\omega$. Then there exists a unique $s\in G$ such that $\beta=\lambda(s)^{*}\sigma\in \spec\, A(G,\omega)\cap G_{\mathbb C,\lambda}\omega$   and
$$
|T_{\sigma}|=T_{\beta}.
$$
\end{lem}
\begin{proof}
Taking the polar decomposition $T_{\sigma}= U |T_{\sigma}|$, we conclude,  as in the proof of \proposition~\ref{complexi}, that $U=\lambda(s)$ for a unique $s\in G$. Clearlywe havw , $\beta=\lambda(s)^*\sigma\in \spec\, A(G,\omega),$ and  $\beta^{*}(\h)\cap\omega^*(\h)=\sigma^{*}(\h)\cap \omega^*(\h)\ne\{0\}$. It follows from  \theorem ~\ref{mainthm}  that the closure of  
$$
\{(\omega \xi,\; \beta \xi)\;|\; \xi\in \h\}\subseteq \h\oplus\h
$$
is the graph of  the positive operator $|T_{\sigma}|$; on the other hand the closure is  the graph of the closed operator  $T_{\beta}.$
\end{proof}
\begin{prop}\label{deform}
Let $\sigma\in \spec\, A(G,\omega)\cap G_{\mathbb C,\lambda}\omega$.  If $\sigma=\lambda(s)\beta$ is the decomposition from \lemma~\ref{polardec}. Then $\psi(t)=\lambda(s)T_{\beta}^{t}\omega$ is a continuous function 
$$
\begin{array}{ccc}
\psi:[0,1]\to \spec \;A(G,\omega)\cap G_{\mathbb C,\lambda}\omega, & \psi(0)=\lambda(s)\omega, & \psi(1)=\sigma.
\end{array}
$$
\end{prop}
\begin{proof}
By the functional calculus, we have $\omega(\h)\subseteq \mathcal{D}(1+T_{\beta})\subseteq \mathcal{D}(T_{\beta}^{t})$ for every $t\in [0,1]$.  Moreover, as $x^{2t}\leq 1+x^{2}$ for $x\in \mathbb{R}_{+},$ we have $$0\leq \omega^{*}T_{\beta}^{2t}\omega\leq \omega^{*}\omega+\omega^{*}T_{\beta}^{2}\omega=\omega^{*}\omega+\beta^{*}\beta.$$ It follows that $T_{\beta}^{t}\omega$ is bounded for every $t\in [0,1]$, and hence  $T_{\beta}^{t}\omega$ and $\lambda(s)T_{\beta}^{t}\omega$ belong to $VN(G)$,  and the function $t\mapsto T_{\beta}^{t}\omega $  is strongly continuous; to see the latter we observe that
if  $P_n=E([0,n])$, where $E(\cdot)$ is the spectral measure of $T_\beta$,  then
$t\mapsto T^t_\beta P_n \omega \xi$ is continuous for every  $\xi\in\h$. Moreover, 
\begin{eqnarray*}
&&\|T_\beta^t(P_n\omega\xi-\omega\xi)\|^2=\langle T_\beta^t(P_n\omega\xi-\omega\xi),T_\beta^t(P_n\omega\xi-\omega\xi)\rangle\\
&&=\langle T_\beta^{2t}(P_n\omega\xi-\omega\xi),(P_n\omega\xi-\omega\xi\rangle\leq \langle T_\beta^2(P_n\omega\xi-\omega\xi),P_n\omega\xi-\omega\xi\rangle\\
&&=\|T_\beta(P_n\omega\xi-\omega\xi)\|^2=\|P_nT_\beta\omega\xi-T_\beta\omega\xi\|^2\to 0
\end{eqnarray*}
 as $P_n\to I$ strongly. 
 Basic approximation arguments give now that $T^t_\beta \omega \xi$ must depend continuously on $t\in[0,1]$ for each $\xi\in\h$.  From this we conclude that $t\mapsto \psi(t)$ is continuous as the map from $[0,1]$ to $VN(G)\simeq A(G,\omega)^*$ with the weak$^*$ topology.

It follows from the functional calculus that $\Gamma(T_{\beta}^{t})=T_{\beta}^{t}\otimes T_{\beta}^{t}$ for all $t\in[0,1],$ and thus $$\Gamma(T_{\beta}^{t}\omega)\Omega=\Gamma(T_{\beta}^{t})(\omega\otimes\omega)=(T_{\beta}^{t}\omega)\otimes (T_{\beta}^{t}\omega)$$ so that $T_{\beta}^{t}\omega$ and hence $\lambda(s)T_{\beta}^{t}\omega$ are  in $\spec\, A(G,\omega).$  

As the kernel of $T_{\beta}$ is trivial, there is  $n\in \mathbb{N}$ such that the orthogonal projection $P=E([\frac{1}{n},n])$ is non-zero. The restriction of $T_{\beta}^{t}$ to the invariant subspace $P\h$ is then invertible for every $t\in [0,1]$  and as $P\lambda(s)^*\psi(t)=PT_\beta^tP\omega$, $t\in [0,1]$, we have
$$\psi(t)^{*}(\lambda(s)P\h)=\omega^{*}(PT_{\beta}^{t}P\h)=\omega^{*}(P\h),$$
giving $\psi(t)^*(\h)\cap\omega^*(\h)\supset\omega^*(P\h)\ne \{0\}$. By \corollary~\ref{spectrum},  we obtain $\psi(t)\in \spec \;A(G,\omega)\cap G_{\mathbb C,\lambda}\omega$ for all $t\in [0,1]$.
\end{proof}

The last results concern  a deformation retraction of weight inverses.  

\begin{lem}\label{sless}
Assume that a weight inverse $\omega$ is positive. For every $s\in [0,1],$ the operator $\omega^{s}$ is again a weight inverse.
\end{lem}
\begin{proof}
By the Löwner-Heinz inequality: if $0\leq A\leq B,$ then also $0\leq A^{s}\leq B^{s}$ for $s\in [0,1]$. Applying this to the inequality~\eqref{iw}, we get
$$
\begin{array}{ccc}
\omega^{2s}\otimes \omega^{2s}\leq \Gamma(\omega^{2s}), & \text{for all $s\in [0,1].$}
\end{array}
$$
The conditions on the kernel(s) is easy to see.
\end{proof}
\begin{prop}
Let $\omega$ be a positive weight inverse. If $\Omega_{s}$ is the $2$-cocycle associated to  $\omega^{s},$ $s\in [0,1]$, then for $0\leq s\leq t\leq1$ the following hold
\begin{enumerate}[$(i)$]
\item
\begin{equation}\label{intertwine}
\Gamma(\omega^{s})\Omega_{t}=\Omega_{t-s}(\omega^{s}\otimes \omega^{s});
\end{equation}
\item if $\ker\;\Omega_{t}^{*}=\{0\},$ then  $\ker\Omega_{s}^*=\{0\}$;
\item the map $\spec\, A(G,\omega^{s})\to\spec\, A(G,\omega^{t})$, given as $\sigma\mapsto \sigma\omega^{t-s}$, is injective and maps  $\spec\, A(G,\omega^{s})\cap G_{\mathbb C,\lambda}\omega^s$ 
to   $ \spec\, A(G,\omega^{t})\cap G_{\mathbb C,\lambda}\omega^t$.
\end{enumerate}
\end{prop}
\begin{proof}
$(i)$ The $2$-cocycle $\Omega_{t-s}$ is the unique operator which satisfies $\Gamma(\omega^{t-s})\Omega_{t-s}=\omega^{t-s}\otimes \omega^{t-s}$. Hence, since $\ker\omega^{s}=\{0\},$  it follows that $\Omega_{t-s}(\omega^{s}\otimes \omega^{s})$ is the unique operator that satisfies $\Gamma(\omega^{t-s})\Omega_{t-s}(\omega^{s}\otimes \omega^{s})=\omega^{t}\otimes \omega^{t}$. As  $\Gamma(\omega^{t-s})\Gamma(\omega^{s})\Omega_{t}=\omega^{t}\otimes\omega^{t},$ we obtain \eqref{intertwine}.

$(ii)$ This follows now directly from~\eqref{intertwine}. 

$(iii)$ If $\sigma\in \spec\,A(G,\omega^{s})$, then by ~\eqref{intertwine} 
$$
\Gamma(\sigma\omega^{t-s})\Omega_{t}=\Gamma(\sigma)\Omega_{s}(\omega^{t-s}\otimes \omega^{t-s}) =(\sigma\omega^{t-s})\otimes (\sigma\omega^{t-s}),
$$
i.e. $\sigma\omega^{t-s}\in \spec\, A(G,\omega^{t})$. 
If $\sigma^{*}(\h)\cap \omega^{s}(\h)\neq \{0\}$ then $$(\omega^{t-s}\sigma)^{*}(\h)\cap (\omega^{t}(\h)=\omega^{t-s}(\sigma^{*}(\h)\cap \omega^{s}(\h))\ne\{0\},$$ as the kernel of $\omega^{t-s}$ is trivial. The injectivity of $\sigma\mapsto\sigma\omega^{t-s}$ follows from the fact that the range of $\omega^{t-s}$ is dense in $\h$. 
\end{proof}


\subsection{\sc{\textbf{Conditions Guaranteeing Complexification}}}
In this section we will investigate conditions on the group $G$ and the weight inverse $\omega$ for which the inclusion (\ref{inclusion}) of the spectrum  of $A(G,\omega)$ into the complexification $G_{\mathbb C,\lambda}$ holds true. 

First, we present  some sufficient conditions for $\ker\;\Omega^{*}=\{0\}$. 

Recall that if $H$ is a closed subgroup of $G$ and  $\lambda_H$ and $\lambda_G$ are the left regular representations of $H$ and $G$ respectively, then there is a canonical injective $*$-homomorphism $\iota_H:VN(H)\to VN(G)$  given by $\lambda_H(s)\mapsto \lambda_G(s)$, for $s\in H$ (\cite{herz}).

We say that a weight inverse $\omega$  on the dual of $G$ is {\it central} if $\omega$ is in the center of $VN(G)$. 
\begin{prop}\label{cond}
Let $\omega$ be a weight inverse on the dual of $G$. Then $\ker\;\Omega^*=\{0\}$ holds provided that any of the following is satisfied:
\begin{enumerate}
\item $G$ is compact;
\item $\omega=\iota_H(\omega_H)$, where $\omega_H$ is a central weight inverse on the dual of a closed subgroup $H$ of $G$.
\end{enumerate}
\end{prop}
\begin{proof}
(1) It is known that if $G$ is compact then $VN(G)\bar\otimes VN(G)\simeq VN(G\times G)$
 can be identified with the $\ell^\infty$ sum of matrix algebras $M_{n_j}(\mathbb C)$. Therefore $\ker\; X=\{0\}\Leftrightarrow \ker \;X^*=\{0\}$ for any $X\in VN(G\times G)$.  This gives $\ker\;\Omega^*=\{0\}$, as $\Omega$ is injective. 

(2) Since $\ker\;\Omega=\{0\}$ it is enough to see that $$\Omega^*\Omega=\Omega\Omega^*,$$ as in this case $\ker\;\Omega=\ker\;\Omega^{*}.$   Assume first that $H=G$. Then being central, $\omega$ is a normal operator and therefore so is $\Gamma(\omega)$. Moreover, as
$\omega\otimes\omega\in Z(VN(G\times G))$ (the center of $VN(G\times G)$),  we have
$$\Gamma(\omega)\Omega\Gamma(\omega)=(\omega\otimes\omega)\Gamma(\omega)=\Gamma(\omega)(\omega\otimes\omega)=\Gamma(\omega)^2\Omega.$$
Hence, as $\ker\;\Gamma(\omega)=\{0\}$, we  have $\Gamma(\omega)\Omega=\Omega\Gamma(\omega)$. By the Fuglede-Putnam theorem it follows that also $\Gamma(\omega)^*\Omega=\Omega\Gamma(\omega)^*$. A calculation now yields
\begin{eqnarray*}
&&\Gamma(\omega)\Omega\Omega^*\Gamma(\omega)^*=\omega\omega^*\otimes\omega\omega^*=
\omega^*\omega\otimes\omega^*\omega\\&&=\Omega^*\Gamma(\omega)^*\Gamma(\omega)\Omega=
\Omega^*\Gamma(\omega)\Gamma(\omega)^*\Omega=\Gamma(\omega)\Omega^*\Omega\Gamma(\omega)^*,
\end{eqnarray*}
and we get the claim by using again  $\ker\;\Gamma(\omega)=\{0\}$.

The proof for general $H$ is  similar, if we take into account that $\Gamma\circ\iota_H=(\iota_H\otimes\iota_H)\circ \Gamma_H$, where $\Gamma_H$ is the comultiplication on $VN(H)$.

\end{proof}

The next simple lemma gives  a sufficient condition for $\sigma^*(\h)\cap\omega^*(\h)\ne\{0\}$  to hold for any $\sigma\in \spec\, A(G,\omega)$, where $\h=L^2(G)$.  We assume that $\ker\;\Omega^{*}=\{0\}.$

\begin{lem}\label{lemma_image}
If there is a subspace $\K\subset \h$ such that $VN(G)(\K)\subset \K$  and $\omega|_{\K}$ is invertible, then
$\sigma^*(\h)\cap\omega^*(\h)\ne\{0\}$ for any $\sigma\in \spec\, A(G,\omega)$.
 \end{lem}

 \begin{proof}
 As $\K$ is invariant and $\omega|_{\K}$ is invertible, $\omega^*(\K)=\K$. We have
 $$\sigma^*(\h)\cap\omega^*(\h)\supset\sigma^*(\K)\cap\omega^*(\K)=\sigma^*(\K),$$
 where the latter is non-zero by \lemma ~\ref{trivcoker}. 
 \end{proof}

Using \lemma~\ref{lemma_image} we can now list groups and weights for which the spectrum of the associated Beur\-ling-Fourier algebra  is in the complexification $G_{\mathbb C,\lambda}$, meaning that we identify $A(G,\omega)^*$ with $VN(G,\omega)$ as in \remark ~\ref{vnGomega}; with a slight abuse of notation we write $\spec\, A(G,\omega)\subset G_{\mathbb C,\lambda}$.
\begin{enumerate}[$(1)$]
\item {\it $G$ is compact and $\omega$ is arbitrary.}
If $G$ is compact then it is known that the left regular representation $\lambda$ on $G$ is a direct sum of irreducible (finite-dimensional) representations and hence there exists a finite-dimensional invariant subspace $\K\subseteq\h$.  As $\ker\omega=\{0\}$, $\omega$ is invertible on $\K$. By \proposition ~\ref{cond}, $\ker\, \Omega^*=\{0\}$. By \corollary ~\ref{spectrum},
$\spec\, A(G,\omega)\subset G_{\mathbb C,\lambda}$. 
In \cite{lst} and \cite{gllst} the result was derived from the "abstract Lie" theory developed in \cite{cartwright_mcmullen, McK1} showing that the multiplicative linear functionals on $\text{Trig}(G)$, the algebra of coefficient functions with respect to  irreducible representations,  can be identified with the complexification $G_{\mathbb C, \lambda}$. As $\text{Trig}(G)\subset A(G,\omega)$, the statement is clear.

\item {\it $G$ is an extension of a compact group by abelian group and $\omega$ is a weight inverse such that  $\ker\,\Omega^*=\{0\}$.}  If $K$ is a non-trivial compact  normal subgroup, let  $P_K\in B(L^2(G))$  be the projection onto the (non-trivial) subspace of functions which are constant on the cosets $xK$, $x\in G$. 
As $P_K$ commutes with $\lambda_G(g)$, $g\in G$, the subspace $P_KL^2(G)$ is invariant with respect to $\lambda_G$, and as $G/K$ is abelian and $P_Kf$ are constant on the cosets, $\lambda_G(g_1g_2)P_Kf=\lambda_G(g_2g_1)P_Kf$, i.e. the von Neumann algebra generated by $\lambda_G(g)P_K$, $g\in G$, is commutative. As $\omega_K:=\omega|_{P_KL^2(G)}$ belongs to the von Neumann algebra,  there exists a subspace $\K$  (e.g. $\K=E_{|\omega_K|}([\varepsilon,\infty))P_KL^2(G)$ for some $\varepsilon>0$)  such that $VN(G)\K\subset\K$ and $\omega|_{\K}$ is invertible.

\item
{\it $G$ is a separable Moore group and $\omega$ is arbitrary.} If $G$ is a Moore group, i.e. any irreducible representation of $G$ is finite dimensional, then $G$ is a type I group with the unitary dual $\hat G$ being a standard Borel space. Moreover, there is a standard Borel measure $\mu$ and a $\mu$-measurable cross section $\xi\to\pi^\xi$ from $\hat G$ to concrete irreducible unitary representation acting on $\h_\xi$ such that $\lambda$ is quasi-equivalent to $\int^{\oplus}_{\hat G}\pi^\xi d\mu(\xi)$ so that $VN(G)\simeq L^\infty (\hat G,d\mu(\xi); B(\h_\xi))$.
With this identification we have $\omega=\int_{\widehat G}^{\oplus}\omega_\xi d\mu(\xi)$. Let for $\varepsilon >0$

$$
\Delta_\varepsilon=\cap_{n}\{\xi\in\widehat G\;|\; \langle|\omega_\xi|x_n(\xi),x_n(\xi)\rangle\geq\varepsilon \|x_n(\xi)\|^2\},
$$
where $(x_n)_n$ is a sequence such that $(x_n(\xi))_n$ is total in $\h_\xi$ for any $\xi$.
As $\ker\omega=\ker\omega^*=\{0\}$, there exists a null set $M\subset\hat G$ such that $\ker\omega_\xi=\ker\omega_\xi^*=\{0\}$ for any $\xi\in\hat G\setminus M$.  Then, as $\h_\xi$ is finite-dimensional, for each $\xi\in\hat G\setminus M$, we have $|\omega_\xi|\;\geq c_\xi I_\xi$ for some $c_\xi>0$. Hence   $\mu(\Delta_\varepsilon)>0$  for some $\varepsilon>0$ and $P_\varepsilon=\int_{\hat G}\chi_{\Delta_\varepsilon}I_\xi\; d\mu(\xi)$ is a non-zero projection onto invariant subspace $\K$ such that $|\omega||_{\K}\geq\varepsilon $; $\omega|_{\K}$ is invertible.  As $\ker\Omega=\{0\}$ and $G\times G$ is Moore, we can argue as above to conclude that $\ker\; \Omega^*=\{0\}$. Therefore, by \corollary ~\ref{spectrum}, we have the inclusion of the spectrum of $A(G,\omega)$ into $G_{\mathbb C,\lambda}$ as  in the previous paragraph.  

    \item {\it $G$ is a separable type $I$  unimodular group and $\omega=\int_{\hat G}^{\oplus}\omega_\xi\; d\mu(\xi)$ with $\omega_\xi$  invertible  on a set ${\mathcal  N}\subset \hat G$ of positive measure. }  We define an invariant subspace $\K$ such that $\omega|_{\K}$ is invertible as above and get the statement of \corollary ~\ref{spectrum} in this case as well if $\ker\; \Omega^*=\{0\}$.  Central weights fall in this class. Any weight on $G$ such that the set $\mathcal N=\{\xi\in\hat G\;|\; \dim\h_\xi<\infty\}$ has positive $\mu$-measure also satisfies that condition.
\end{enumerate}

Recall that a locally compact group $G$ is called an {\it [IN]-group} if it has a compact conjugation-invariant neighbourhood of the identity.  It is called a {\it [SIN]-group} if it has a base of conjugate-invariant neighbourhoods of $e$. We note that any [SIN]-group is [IN]. Typical [SIN]-groups are discrete, compact and abelian groups.

The following result is likely known, but we could not find a reference.
\begin{prop}
$G$ is an [IN]-group  if and only if $VN(G)$ admits a normal tracial state. 
\end{prop}
\begin{proof}
Assume that $\text{tr}\in VN(G)_{*}$ is a tracial state. Then the function $f(g)=\text{tr}(\lambda(g))\in A(G)\subseteq C_{0}(G)$, and thus we have a conjugate-invariant compact neighbourhood, e.g.
$$
\{g\in G\,|\, |f(g)|\geq \frac{1}{2}\}.
$$
Conversely, assume that $K$ is a compact neighbourhood which is conjugate-invariant, and let $\xi_K$ be the $L^{2}$-normalized indicator function of $K.$ Consider the state $\phi(\cdot )=\langle\cdot\xi_K,\xi_K\rangle\in VN(G)_{*}.$ Then we have for all $g,h\in G$ (note that [IN] groups are unimodular)
$$
\phi(\lambda(gh))=\langle \lambda(gh)\xi_K,\xi_K\rangle=\int_{G} \xi_K(h^{-1}g^{-1}x)\overline{\xi_K(x)} dx=\int_{G}\xi_K(h^{-1}x)\overline{\xi_K(gx)} dx=
$$
$$
=\int_{G}\xi_K(h^{-1}x)\overline{\xi_K(xg)} dx=\int_{G}\xi_K(h^{-1}xg^{-1})\overline{\xi_K(x)}dx=\int_{G}\xi_K(g^{-1}h^{-1}x)\overline{\xi_K(x)}dx=
$$
$$
=\langle \lambda(hg)\xi_K,\xi_K\rangle=\phi(\lambda(hg)).
$$
Thus $\phi$ is tracial.
\end{proof}
\begin{cor}
If $G$ is an [IN]-group and $\omega\in VN(G)$ is a weight inverse such that $\ker\, \Omega^{*}=\{0\}$, then $\spec\, A(G,\omega)\subset G_{\mathbb{C},\lambda}.$
\end{cor}
\begin{proof}
As $\ker\,\Omega^{*}=\{0\}$ if follows from \lemma~\ref{trivcoker} that for all $\sigma\in \spec\,A(G,\omega)$ we have  $\sigma(\mathcal{H})$ is dense in $\mathcal{H}.$
Consider the following two inequalities 
\begin{equation}\label{sw}
\begin{array}{ccc}
\sigma^{*}\sigma\leq \sigma^{*}\sigma+\omega^{*}\omega, & \omega^{*}\omega\leq \sigma^{*}\sigma+\omega^{*}\omega.
\end{array}
\end{equation}
Letting $R=(\sigma^{*}\sigma+\omega^{*}\omega)^{\frac{1}{2}}$, we can deduce from~\eqref{sw} -- similar to the proof of \lemma~\ref{lem1}-- that there exist  $U,V\in VN(G)$ such that 
\begin{equation}\label{ur}
\begin{array}{ccc}
U R=\sigma,& VR=\omega.
\end{array}
\end{equation}
Moreover, we have
$$
R(U^{*}U+V^{*}V)R=\sigma^{*}\sigma+\omega^{*}\omega=R^{2},
$$
so that the density of the range of $R$ (implied by the density of the range of $\omega$) gives 
\begin{equation}\label{kuntz}
U^{*}U+V^{*}V=I.
\end{equation}
In particular, we obtain  that $U^{*}U$ and $V^{*}V$ commute. 

Assume towards contradiction that $\sigma^{*}(\mathcal{H})\cap \omega^{*}(\mathcal{H})=\{0\}.$ Then by~\eqref{ur} and the injectivity of $R$ we can deduce that also $U^{*}(\mathcal{H})\cap V^{*}(\mathcal{H})=\{0\}.$ Thus 
$$(U^{*}U)(V^{*}V)=(V^{*}V)(U^{*}U)=0,$$
so that~\eqref{kuntz} implies that $U,V$ are partial isometries. As $\ker\,\Omega^{*}=\{0\}$, it follows from \lemma~\ref{trivcoker} that $\ker \sigma^{*}=\ker \omega^{*}=\{0\}$, and by~\eqref{ur}, $\ker U^{*}=\ker V^{*}=\{0\}$. Thus $U^{*}$ and $V^{*}$ are isometries in $VN(G)$ such that~\eqref{kuntz} holds,  i.e. $(U,V)$ is a representation of the Cuntz algebra $O_2$ in $VN(G).$ This contradicts the claim that $VN(G)$ admits a tracial state $\phi$:
$$
1=\phi(I)=\phi(U^{*}U+V^{*}V)=\phi(U^{*}U)+\phi(V^{*}V)=\phi(UU^{*})+\phi(VV^{*})=2.
$$
\end{proof}


\begin{prop}\label{sin}
If $G$ is a [SIN]-group then $\spec\, A(G,\omega)\subset G_{\mathbb{C},\lambda}$ for any weight inverse $\omega$. 
\end{prop}

\begin{proof}
By \cite[13.10.5]{dixmier}, $G$ is a  [SIN]-group if and only if $VN(G)$ is finite. Therefore, as $\ker\; \Omega=\{0\}$, we have $\ker\; \Omega^*=\{0\}$, giving, by \lemma~\ref{trivcoker}, $\ker\sigma^*=\{0\}$ and hence by finiteness of $VN(G)$, $\ker\sigma=\{0\}$ for any $\sigma\in \spec A(G,\omega)$. As both $(\sigma^*)^{-1}$ and $(\omega^*)^{-1}$ are densely defined and affiliated with $VN(G)$, and the set of affiliated elements is an algebra, we obtain that $\sigma^*(L^2(G))\cap\omega^*(L^2(G))\ne\{0\}$ as the domain of $(\sigma^*)^{-1}+(\omega^*)^{-1}\in \overline{VN(G)}$. Therefore, $\spec A(G,\omega)\subset G_{\mathbb C,\lambda}$, by \corollary~\ref{spectrum}.
\end{proof}

We remark  that $VN(G)$ is finite for all Moore groups $G$ and hence any such $G$ is [SIN].

\begin{cor}
If $G$ is discrete, then $\spec\, A(G,\omega)=G$ for any weight inverse $\omega$. 
\end{cor}

\begin{proof}
$G$ clearly does not contain any non-trivial image of a homomorphism $\mathbb{R}\to G,$ and we can deduce that the complexification is trivial, i.e. $G_{\mathbb{C},\lambda}=G.$ Moreover, as $G$ is a [SIN]-group, by \proposition~\ref{sin}  the spectrum of $A(G,\omega)$ is the smallest possible, that is $G$. 
\end{proof}
\medskip

An important class of weights that has been studied in the literature are weights extended from closed abelian  or compact subgroups, see \cite[Proposition 3.25]{gllst}. 
The next statements show that for all such weights we have the inclusion  of the spectrum  into the complexification.  We first recall the construction of a so called central weight on the dual of a compact group following  \cite{gllst}, see also \cite{lst}. 

If  $H$ is compact, we have the quasi-equivalence $\lambda\simeq\oplus_{\pi\in\widehat  H}\pi$ which gives $VN(H)\simeq\oplus_{\pi\in \widehat H} M_{d_{\pi}}$, where $d_\pi$ is the dimension of the representation space $H_\pi$. We have also the Plancherel theorem giving the isomorphism 
$$L^2(H)\simeq\oplus_{\pi\in\widehat H}^{\ell^2}\sqrt{d_\pi}S_{d_\pi}^2$$ with
$\langle \xi,\eta\rangle=\sum_{\pi\in\widehat H} d_\pi\text{tr}(\hat\xi(\pi)\hat\eta(\pi)^*)$, for $\xi$, $\eta\in L^2(H)$, where $S_{n}^2$ reffers to Hilbert-Schmidt class on $\ell^2_n$  and $\hat\xi(\pi)=\int_H\xi(s)\pi(s^{-1})ds$. 
Recall also that for $A = (A(\pi))_{\pi\in \widehat{H}} \in VN(H)$ we have
	$$\Gamma(A) = \oplus_{\pi, \pi'}\left[ U^*_{\pi,\pi'}(\oplus_{\sigma \subseteq \pi\otimes \pi'} A(\sigma))U_{\pi, \pi'} \right]$$
where for $\sigma, \pi,\pi'\in\widehat{H}$, the notation $\sigma \subseteq \pi\otimes \pi'$ means that $\sigma$ is a subrepresentation of $\pi\otimes \pi'$, and $U_{\pi, \pi'}$ is the unitary appearing in the irreducible decomposition of $\pi \otimes \pi'$.

If $\omega_H$ is a central positive weight inverse then  $\omega_H\simeq \oplus_{\pi\in \widehat H}\omega(\pi)I_\pi$ for a function $\omega:\widehat H\to(0,+\infty)$ which satisfies
$\omega(\pi)\omega(\rho)\leq\omega(\sigma)$ for any $\sigma$, $\pi$, $\rho\in \widehat H$ such that $\sigma\subseteq\pi\otimes\rho$, see \cite[3.3.2]{gllst}. We refer the reader to \cite{lst} and \cite[section 5]{gllst} for numerous examples of central weight inverses. 

Recall the conjugate representation  $\bar\pi$ of $\pi\in \widehat H$ which is defined as follows: we denote the linear dual space of $H_\pi$ by $H_{\bar\pi}$ and for $A\in B(H_\pi)$, let $A^t$ in $B(H_{\bar\pi})$ be its linear adjoint; for $s\in H$ we define $\bar\pi(s)=\pi(s^{-1})^t$; it is a unitary irreducible representation on $H_{\bar\pi}$ and $\bar{\bar\pi}=\pi$, as equivalence classes. 

For the antipode $S$ and the central weight $\omega_H$ we have  $S(\omega_H)\simeq\oplus_{\pi\in \widehat H}\omega(\bar\pi)I_\pi$. Indeed, we observe first that
$$\langle S(\omega_H)\xi,\eta\rangle=\langle\omega_H\bar\eta,\bar\xi\rangle=\sum_{\pi\in\widehat H}\omega(\pi)\text{tr}(\hat{\bar\eta}(\pi)\hat{\bar\xi}(\pi)^*).$$
Because of the unitary equivalence  $\hat{\bar\xi}(\pi)^*=\int_H\xi(s)\pi(s)ds\sim\int_{ H}\xi(s)\bar\pi(s^{-1})^tds$ and $\hat{\bar\eta}(\pi)=\int_H\bar\eta(s)\pi(s^{-1})ds\sim\int_H\bar\eta(s)(\bar\pi(s^{-1})^*)^tds$ with the same unitary operator, we obtain
$$\langle S(\omega_H)\xi,\eta\rangle=\sum_{\pi\in\widehat H}\omega(\pi)\text{tr}(\hat{\xi}(\bar\pi)\hat{\eta}(\bar\pi)^*)=\sum_{\pi\in\widehat H}\omega(\bar\pi)\text{tr}(\hat{\xi}(\pi)\hat{\eta}(\pi)^*),$$
that shows the statement.

If  $\sigma\subseteq\pi\otimes\rho$ then  $\rho\subseteq\bar\pi\otimes\sigma$ which  follows from 
\cite[2.30, 2.34(b,c)]{hewitt_ross}; this  gives $\omega(\sigma)\omega(\bar\pi)\leq\omega(\rho)$ which together with the expression for the antipode  gives the inequality
 \begin{equation}\label{swift}
 (S(\omega_H)\otimes I)\Gamma(\omega_H)\leq I\otimes\omega_H,
 \end{equation}
 the arguments  for this are similar to that given in \cite[3.3.2]{gllst}.
 
 If $H$ is abelian,  the weight inverse condition for positive weight inverse  $\omega_H$ can be also equivalently written  as ~\eqref{swift} since in this case
$$
\begin{array}{ccc}
 \omega_{H}(s^{-1})\omega_{H}(st)\leq \omega_{H}(t) & \text{ for almost all }s,t\in \widehat H.
\end{array}
$$

\begin{thm}\label{very}
Let $H\subseteq G$ be a closed abelian or compact subgroup of $G$ and $\iota_H:VN(H)\to VN(G)$ be the canonical injective homomorphism. 
Let $\omega_{H}\in VN(G)$ be a central weight inverse. If $\omega:=\iota_H(\omega_{H})$, 
then every $\sigma\in \spec\, A(G,\omega)$ is a weight inverse and \begin{equation}\label{samesame}S(\sigma)\sigma=S(\omega)\omega.\end{equation}
Moreover, 
$\spec\, A(G,\omega)\subset G_{\mathbb{C},\lambda}.$
\end{thm}

\begin{proof} Let $\Omega$ be the $2$-cocycle associated to $\omega$. 
By \proposition~\ref{cond}  $\ker \;\Omega^{*}=\{0\}$ and hence $\ker\sigma^{*}=\{0\} $ for every $\sigma\in \spec\, A(G,\omega)$ by \lemma ~\ref{trivcoker}. To show that $\sigma$ is a weight inverse, it is enough to see the equality~\eqref{samesame}, which will  imply $\ker\sigma=\{0\}.$ To prove~\eqref{samesame}, let us without any loss of generality assume that $\omega_{H}$ is positive. Then $\omega_H$ satisfies ~\eqref{swift}. 
 It is clearly  preserved by $\iota_H$  giving $$(S(\omega)\otimes I)\Gamma(\omega)\leq I\otimes \omega. $$ As in the proof of \lemma ~\ref{lem1} we can conclude that there is a unique element $\Phi\in VN(G)\bar\otimes VN(G)$ such that
$$
(I\otimes \omega)\Phi=(S(\omega)\otimes I)\Gamma(\omega).
$$
If $W$ is the fundamental multiplicative unitary, the latter equality gives
\begin{equation}\label{relation}
(I\otimes\omega)\Phi W^*=(S(\omega)\otimes I)W^*(I\otimes\omega). 
\end{equation}
Let $\xi,\eta, \tilde\xi, \tilde\eta\in \h.$ We retain the notation $\psi_{x,y}$ for the normal functional $\psi_{x,y}(T)=\langle Tx,y\rangle$, $T\in B(\h)$.
By \lemma~\ref{star_lemma},
\begin{equation}\label{star1}
\psi_{\xi,\tilde\xi}(S(\iota\otimes\psi_{\sigma^*\eta,\tilde\eta}(\Omega^*W^*)))=
\langle(1\otimes\sigma^*)W(S(\sigma^*)\otimes 1)\xi\otimes\eta,\tilde\xi\otimes\tilde\eta\rangle,
\end{equation}
for any $\sigma\in \spec\, A(G,\omega)$. In particular, it holds for $\omega$ which combined with (\ref{relation}) gives
\begin{equation}\label{star2}
\psi_{\xi,\tilde\xi}(S(\iota\otimes\psi_{\omega^*\eta,\tilde\eta}(\Omega^*W^*)))=
\langle W\Phi^*\xi\otimes\omega^*\eta,\tilde\xi\otimes\tilde\eta\rangle,
\end{equation}
Fix $\sigma\in \spec\, A(G,\omega)$. As the range of $\omega^*$ is dense in $\h$, there exists $\{\eta_n\}_n\subset\h$ such that $\omega^*\eta_n\to\sigma^*\eta$.  From (\ref{star1}) and (\ref{star2})  we obtain
\begin{eqnarray*}
&&\langle(1\otimes\sigma^*)W(S(\sigma^*)\otimes 1)\xi\otimes\eta,\tilde\xi\otimes\tilde\eta\rangle=\psi_{\xi,\tilde\xi}(S(\iota\otimes\psi_{\sigma^*\eta,\tilde\eta}(\Omega^*W^*)))\\&&=\lim_{n\to\infty}\psi_{\xi,\tilde\xi}(S(\iota\otimes\psi_{\omega^*\eta_n,\tilde\eta}(\Omega^*W^*)))=\lim_{n\to\infty}\langle W\Phi^*\xi\otimes\omega^*\eta_n,\tilde\xi\otimes\tilde\eta\rangle\\ &&=\langle W\Phi^*\xi\otimes\sigma^*\eta,\tilde\xi\otimes\tilde\eta\rangle\end{eqnarray*}
giving
\begin{equation}\label{wiry}
(I\otimes\sigma) \Phi W^{*}=(S(\sigma)\otimes I)W^{*}(I\otimes\sigma).
\end{equation}
Reasoning as in the remark after the proof of \proposition~\ref{thm1}, we have that the operator $$M:=(\Phi W^*) W \Omega$$ satisfies
\begin{eqnarray}\label{sigmasigma}
(I\otimes \sigma)M&=&(S(\sigma)\otimes I)W^{*}(I\otimes\sigma)W\Omega\\&=&(S(\sigma)\otimes I)\Gamma(\sigma)\Omega=
S(\sigma)\sigma\otimes \sigma,\nonumber
\end{eqnarray}
for every $\sigma\in \spec\, A(G,\omega)$. As $\omega\in \spec\, A(G,\omega) $ and  $\ker \omega=\{0\},$ it tells that  $M=S(\omega)\omega\otimes I$ and $(I\otimes\sigma)M=(S(\omega)\omega)\otimes \sigma$ which together with (\ref{sigmasigma}) give the equality 
$S(\omega)\omega=S(\sigma)\sigma$.
It now  implies $\sigma^{*}(\h)\cap\omega^{*}(\h)\neq\{0\}$  and hence by \proposition~\ref{cond} and  \corollary~\ref{spectrum}, we get the claimed inclusion for the spectrum.
\end{proof}

\begin{thm}\label{reduction}
Let $H\subseteq G$ be a closed subgroup, 
$\omega_H\in VN(H)$  be a weight inverse 
and $\omega:=\iota_H(\omega_H)$. 
Assume that 
$\spec A(G,\omega)\subset G_{\mathbb C ,\lambda}$. Then every
$\sigma\in \spec\, A(G,\omega)$ is of the form 
$$
\sigma=\lambda_G(s)\iota_H(\tilde{\sigma}),
$$
for some $s\in G$ and $\tilde{\sigma}\in \spec \, A(H,\omega_{H}).$
\end{thm}

\begin{proof} The condition 
$\spec A(G,\omega)\subset G_{\mathbb C,\lambda}$ implies that any $\sigma\in \spec A(G,\omega)$ admits a factorisation  $\sigma=T\omega$ for some $T\in G_{\mathbb C,\lambda}$; 
hence $\sigma^*\supset \omega^*T^*$ showing that $\sigma^*L^2(G)\cap\omega^*L^2(G)\ne\{0\}$. By \proposition~\ref{thm1}, we get $S(\sigma)\sigma=S(\omega)\omega=\iota_H(S(\omega_H)\omega_H)\in\iota_H(VN(H))$. Moreover,  as $T^*=(\sigma\omega^{-1})^*=(\omega^{-1})^*\sigma^*$ and $\ker T^*=\{0\}$, $\ker\sigma^*=\{0\}$ showing that 
$\sigma$ is a weight inverse. It follows that
\begin{equation}\label{inH}
\Gamma(S(\sigma)\sigma)\Omega=\Gamma(S(\sigma))(\sigma\otimes\sigma)\in (\iota_H\bar\otimes \iota_H)(VN(H)\bar\otimes VN(H) ),
\end{equation}
and is independent of particular $\sigma\in \spec\, A(G,\omega).$ Applying a slice map $\iota\otimes f,$  $f\in A(G)$, to~\eqref{inH}  and using the fact that  the elements of the form 
$f\sigma$ form a  dense subspace in $A(G)$  (as the range of $\sigma$ is dense) we obtain
\begin{equation}\label{cut}
\begin{array}{ccc}
(\iota \otimes f)(\Gamma(S(\sigma)))\sigma\in \iota_H(VN(H)), & \text{for all $f\in A(G).$}
\end{array}
\end{equation}
Consider the subspace $$\mathcal A=\overline{\{(\iota \otimes f)(\Gamma(S(\sigma)))\;|\;f\in A(G)\}}^{w^*}\subseteq VN(G)$$ (the weak$^{*}$ closure). By~\eqref{cut}, we have $\mathcal A\sigma\subseteq \iota_H(VN(H)).$ Let $I_{\mathcal A}\subseteq A(G)$ be the preannihilator  of $\mathcal A$, i.e.
$$I_{\mathcal A}=\mathcal A_{\perp}:=\{f\in A(G)\;|\; \langle A, f\rangle =0\ \forall A\in\mathcal A\}.$$ We claim that $I_{\mathcal A}$ is equal to the subspace
$$
\{f\in A(G)\;|\; (\iota\otimes f)(\Gamma(S(\sigma)))=0\},
$$
and indeed, this follows from the action of $A(G)$ on $VN(G)$ being commutative. Moreover, the same argument shows that $I_{\mathcal A}\subseteq A(G)$ is a non-trivial closed ideal, as $\sigma\neq 0$. By duality, we have $$\mathcal A=(\mathcal A_\perp)^\perp=\{x\in VN(G)\;|\; f(x)=0,\, \forall f\in I_{\mathcal A} \}.$$ As $I_{\mathcal A}\neq A(G)$, there is at least one $s\in G$ such that $\lambda_G(s)^{*}$ annihilates $I_{\mathcal A},$ and hence $\lambda_G(s)^{*}\in \mathcal A.$ It follows that 
$$
\lambda_G(s)^{*}\sigma\in \iota_H(VN(H)),
$$
and moreover that the pre-image $\tilde{\sigma}=\iota_H^{-1}(\lambda_G(s)^{*}\sigma)\in \spec\, A(H,\omega_{H}).$ This gives the statement of the theorem.
\end{proof}

Combining methods in the proofs of \theorem~\ref{very} and \theorem~ \ref{reduction} we obtain a generalisation of \theorem~\ref{very} to weights induced from  non-central weights of  compact subgroups of $G$. 
\begin{thm}\label{veryext}
Let $H\subseteq G$ be a compact subgroup, and $\omega_H\in VN(H)$ be a weight inverse on the dual of $H$. 
 Then with $\omega=\iota_{H}(\omega_H),$ we have $$\spec\,A(G,\omega)\subseteq G_{\mathbb{C},\lambda}.$$ Moreover, every $\sigma \in \spec\,A(G,\omega)$ is of the form $\lambda_G(s)\iota_{H}(\tilde{\sigma})$ for some $s\in G$ and $\tilde\sigma \in \spec\,A(H,\omega_H).$
\end{thm}
\begin{proof}
 Let $F\subset \widehat H$ be finite and set $\tilde P_F$ to be the central projection in $VN(H)$ given by $\tilde P_F=\oplus_{\pi\in\hat H}\chi_F(\pi)I_{\pi}$, where $\chi_F$ is the indicator function of $F$. Set
 $$C_F=\{\pi\in\widehat H\,|\,\pi\subseteq\pi_1\otimes\pi_2, \pi_1,\pi_2\in F\}.$$
 Then using arguments as in \cite[3.3.2]{gllst}, we obtain  $\tilde P_F\otimes \tilde P_F\leq\Gamma(\tilde P_{C_F})$ and hence
 $(\tilde P_F\otimes \tilde P_F)\Gamma(\tilde P_{C_F})=\tilde P_F\otimes \tilde P_F$, which gives
 $$(\tilde P_F\otimes \tilde P_F)W^*(I\otimes \tilde P_{C_F})=(\tilde P_F\otimes \tilde P_F)W^*.$$
 As $(\tilde P_F\otimes \tilde P_F)W^*\in VN(H)\otimes\left(\oplus_{\pi\in F}M_{d_\pi}\right)$ we can apply $S\otimes\iota$ to the last equality to obtain
 $$
(I\otimes \tilde P_{F})W(S(\tilde P_{F})\otimes \tilde P_{C_F})=(I\otimes \tilde P_{F})W(S(\tilde P_{F})\otimes I) 
$$
and
\begin{equation}\label{equ}
\Gamma(\tilde P_{F})(S(\tilde P_{F})\otimes \tilde P_{C_F})=\Gamma(\tilde P_{F})(S(\tilde P_{F})\otimes I).
\end{equation}
 
In $\overline{VN(H\times H)}$, we consider the element
$$
\tilde\Phi^{*}=\Gamma(\omega_H^{*})(S( \omega_H^{*})\otimes I) (I\otimes{(\omega_H^{*})}^{-1}).
$$
Note that for $\xi\in L^2(H)$, $\eta\in \mathcal D(( \omega_H^*)^{-1})$,
\begin{eqnarray*}
\Gamma(\tilde P_{F})\tilde\Phi^{*}(S(\tilde P_{F})\xi\otimes \eta)&=&\Gamma(\tilde P_{F})\Gamma(\omega_H^{*})(S( \omega_H^{*})\otimes I)(S(\tilde P_{F})\xi\otimes (\omega_H^*)^{-1}\eta)\\
&=&\Gamma(\omega_H^{*})\Gamma(\tilde P_{F})(S(\tilde P_{F})\otimes I)(S(\tilde \omega^{*})\xi\otimes {(\omega_H^{*})}^{-1}\eta)\\
&=&\Gamma(\omega_H^{*})\Gamma(\tilde P_{F})(S(\tilde P_{F})\otimes \tilde P_{C_F})(S(\omega_H^{*})\xi\otimes {(\omega_H^{*})}^{-1}\eta)
\\
&=&\Gamma(\omega_H^{*} \tilde P_{F})((S(\omega_H^{*}\tilde P_{F})\xi\otimes {(\omega_H^{*})}^{-1}\tilde P_{C_F}\eta)
\end{eqnarray*}
from which we conclude that $\Gamma(\tilde P_{F})\tilde\Phi^{*}(S(\tilde P_{F})\otimes I)$ extends to a bounded operator in $VN(H\times H).$ 

Now let $P_{F}=\iota_{H}(\tilde P_{F})$, $P_{C_F}=\iota_{H}(\tilde P_{C_F})$ and $\Phi^{*}=(\iota_{H}\otimes\iota_H)(U)(\iota_H\otimes\iota_H)(|\tilde \Phi^{*}|)\in\overline{VN(G\times G)},$ where $\tilde\Phi^*=U|\tilde\Phi^*|$ is the polar decomposition of $\tilde\Phi^*$, and $(\iota_H\otimes\iota_H)(|\tilde \Phi^{*}|)$ is the extension of $\iota_H\otimes\iota_H$ to the positive operator $|\tilde \Phi^{*}|\in \overline{VN(H\times H)}$, see \cite[Section 2]{gllst}.  

 Applying \cite[\proposition~2.1]{gllst} we can conclude that $|\tilde \Phi^*|(I\otimes \tilde\omega^*)\in VN(H\times H)$ and  $\iota_H\otimes\iota_H(|\tilde \Phi^*|)(I\otimes \omega^*)\in VN(G\times G)$, which show that $\Phi^*(I\otimes\omega^*)$ is bounded and 
 
\begin{equation}\label{relation}
\Phi^{*}(I\otimes \omega^{*})=\Gamma(\omega^{*})(S(\omega^{*})\otimes I).
\end{equation}
Let $\Omega$ be the 2-cocycle associated with $\omega$. As in the proof of \theorem~\ref{very}  we have for $\sigma\in  \spec\, A(G,\omega)$
\begin{equation}\label{star1}
\psi_{\xi,\tilde\xi}(S(\iota\otimes\psi_{\sigma^*\eta,\tilde\eta}(\Omega^*W^*)))=
\langle(1\otimes\sigma^*)W(S(\sigma^*)\otimes 1)\xi\otimes\eta,\tilde\xi\otimes\tilde\eta\rangle,
\end{equation} and 
\begin{equation}\label{star2}
\psi_{\xi,\tilde\xi}(S(\iota\otimes\psi_{\omega^{*}\eta,\tilde\eta}(\Omega^*W^*)))=
\langle W\Phi^*(\xi\otimes\omega^{*}\eta),\tilde\xi\otimes\tilde\eta\rangle,
\end{equation} where
 $\xi,\eta, \tilde\xi, \tilde\eta\in \h$.
 
Take  $\xi\in S(P_{F})L^{2}(G)$ and $\tilde \eta\in P_{F} L^{2}(G)$. Then the right-hand side of~\eqref{star2} becomes 
$$
\langle W(\Gamma(P_{F})\Phi^*(S(P_{F})\otimes I))(\xi\otimes\omega^{*}\eta),\tilde\xi\otimes\tilde\eta\rangle.
$$
Fix $\sigma\in \spec\, A(G,\omega)$. As the range of $\omega^{*}$ is dense in $\h$, there exists $\{\eta_n\}_n\subset\h$ such that $\omega^{*}\eta_n\to\sigma^*\eta$.  From (\ref{star1}) and (\ref{star2}) together with  $\Gamma(P_{F})\Phi^*(S(P_{F})\otimes I)\in VN(G\times G),$ we get
\begin{eqnarray*}
&&\langle(1\otimes\sigma^*)W(S(\sigma^*)\otimes 1)\xi\otimes\eta,\tilde\xi\otimes\tilde\eta\rangle=\psi_{\xi,\tilde\xi}(S(\iota\otimes\psi_{\sigma^*\eta,\tilde\eta}(\Omega^*W^*)))\\&&=\lim_{n\to\infty}\psi_{\xi,\tilde\xi}(S(\iota\otimes\psi_{\omega^*\eta_n,\tilde\eta}(\Omega^*W^*)))=\lim_{n\to\infty}\langle W\Gamma(P_{F})\Phi^*(S(P_{F})\otimes I)\xi\otimes\omega^*\eta_n,\tilde\xi\otimes\tilde\eta\rangle\\ &&=\langle W\Gamma(P_{F})\Phi^*(S(P_{F})\otimes I)
\xi\otimes\sigma^*\eta,\tilde\xi\otimes\tilde\eta\rangle\end{eqnarray*}
giving
\begin{equation}\label{wiry}
(\Gamma(P_{F})\Phi^{*}(S(P_{F})\otimes I))(I\otimes \sigma^{*}) =\Gamma(P_F)\Gamma(\sigma^{*})(S(\sigma)^{*}S(P_F)\otimes I).
\end{equation}

Let $M$ be in the commutant of $\iota_{H}(VN(H))\bar\otimes VN(G)$.
Clearly, $M$ commutes with the left-hand side  of (\ref{wiry}) and as $\tilde P_F\to I$ weak$^*$, we obtain that it commutes with $\Gamma(\sigma^{*})(S(\sigma)^{*}\otimes I)$. Therefore,  
\begin{equation}\label{mmm}
\Gamma(\sigma^{*})(S(\sigma)^{*}\otimes I)\in \iota_{H}(VN(H))\bar\otimes VN(G).
\end{equation}
If we let $f\in VN(G)_{*}$ be arbitrary, then it follows from~\eqref{mmm}
$$
S(\sigma)(\iota\otimes f)(\Gamma(\sigma))\in \iota_{H}(VN(H)).
$$
We now proceed in a similar way as in the proof of \theorem~\ref{reduction} and let $$\mathcal{A}=\overline{\{(\iota\otimes f)(\Gamma(\sigma))\,|\, f\in A(G)\}}^{w^*}.$$ We can argue as before that the ideal $I_{\mathcal A}:=\mathcal A_{\perp}\ne A(G)$ and hence there is $s\in G$ such that  $f(s)=0$ for all $f\in I_{\mathcal A}$. As $I_{\mathcal A}^{\perp}=\mathcal A$, $\lambda_G(s)\in \mathcal A$ and therefore $S(\sigma)\lambda_G(s)\in \iota_H(VN(H))$ and $\lambda_G(s^{-1})\sigma\in \iota_H(VN(H))$.  It follows that there is an $\tilde\sigma\in\spec \, A(H,\omega_H)$ such that $\lambda_G(s^{-1})\sigma=\iota_{H}(\tilde\sigma)$, and hence 
\begin{equation}\label{eq}
\sigma=\lambda_G(s)\iota_{H}(\tilde\sigma).
\end{equation}
As $\spec\, A(H,\omega_H)\subseteq H_{\mathbb{C},\lambda},$ we conclude  that $\spec\,A(G,\omega)\subseteq G_{\mathbb{C},\lambda}$.

\end{proof}

Let $G$ be a connected simply connected  Lie group and $\mathcal g$ its associated Lie algebra. We also fix the symbol $H$  and $\mathcal h$ for a connected closed Lie subgroup of $G$ and its Lie algebra respectively. We write $\lambda_G$ and $\lambda_H$ for the left regular representations of $G$ and $H$ respectively. 
The next statement generalizes \cite[\theorem~5.9, \theorem~6.19, \theorem~7.11, \theorem~8.20 and \theorem~9.11]{gllst}, where it was proved for compact connected Lie groups with a weight induced from a closed Lie subgroup, the Heisenberg group,  the reduced Heisenberg group, the Euclidean motion group on $\mathbb R^2$, and the simply connected cover of it  with weights induced from abelian connected Lie subgroups. We note that the proofs of the latter theorems from \cite{gllst} required lengthy and specific  arguments for each particular group. We also answer \cite[Question 11.4]{gllst} as our technique does not require the existence and density of entire vectors for the left regular representation which was essential to prove the mentioned results in \cite{gllst}. 

\begin{thm}\label{lie}
Let $G$ be a connected simply connected Lie group and let $H$ be either abelian or compact connected closed subgroup of $G$. Suppose $\omega_H$ is a positive weight inverse on the dual of $H$ and $\omega=\iota_H(\omega_H)$ is the extended weight  inverse on the dual of $G$. Then 
$$\spec\, A(G,\omega)\simeq \{\lambda_G(s)\exp{i\partial\lambda_G(X)}\;|\;s\in G, X\in \mathcal h, \exp{i\partial\lambda_H(X)}\in \spec\, A(H,\omega_H)\}.$$
\end{thm}
\begin{proof}
By \theorem~\ref{very}, \theorem~\ref{veryext} and the remark after \lemma~\ref{lemma_image} we have
$$\spec\,A(G,\omega)\subset G_{\mathbb C,\lambda},$$ and hence by  \proposition~\ref{factorisation} for any $\sigma\in \spec\, A(G,\omega)$ there is a unique $s\in G$ and $X\in \mathcal g$ such that $\text{Ran}(\omega)\subset\mathcal D(\exp i\partial\lambda_G(X))$, $\exp i\partial\lambda_G(X)\omega$ is bounded and $$(\sigma,u)=(\lambda_G(s)\exp i\partial\lambda_G(X)\omega,u)_\omega$$ for all $u\in A(G,\omega)$.
By \theorem~\ref{reduction} and \theorem~\ref{veryext},  we have that $$\lambda_G(s)\exp i\partial\lambda_G(X)\omega=\lambda_G(t)\iota_H(\tilde\sigma)$$ for some $\tilde\sigma\in \spec\, A(H,\omega_H)$ and 
$t\in G$.   By assumption of the theorem, there exist $\tilde s\in H$ and  $\tilde X\in\mathcal h$  such that and $\tilde\sigma=\lambda_H(\tilde s)\exp i\partial\lambda_H(\tilde X)$. As $\iota_H( \exp i\partial\lambda_H(\tilde X))=\exp i\partial\lambda_G(\tilde X)$, we obtain by applying \cite[\proposition~ 2.1]{gllst} that $\iota_H(\tilde\sigma)=\lambda_G(\tilde s)\exp i\partial\lambda_G(\tilde X)$  from which we get  the inclusion "$\subset$". 

Conversely, if $\exp i\partial\lambda_H(X)\in \spec\, A(H,\omega_H)$, then $\exp{i\partial\lambda_G(X)}\in \spec\, A(G,\iota_H(\omega_H))$, which follows from \cite[\proposition~ 2.1]{gllst}. 
\end{proof}

\begin{example}\rm 
Consider the $"ax+b"$-group  that can be represented as the group $G$ of matrices:
$$G=\left\{g=\left(\begin{array}{cc}a&b\\0&a^{-1}
\end{array}\right)\left| \right. a>0, b\in\mathbb R\right\}.$$
It is known to be the semidirect product of the subgroups
$$A=\left\{\left(\begin{array}{cc}a&0\\0&a^{-1}
\end{array}\right) \left | \right. a>0 \right\}\text{ and } B=\left\{\left(\begin{array}{cc}1&b\\0&1
\end{array}\right)\left | \right. b\in\mathbb R\right\}.$$
The Lie algebra of $G$ is generated by $H$ and $E$ given by 
$$H=\left(\begin{array}{cc}1/2&0\\0&-1/2\end{array}\right)\text{ and }E=\left(\begin{array}{cc}0&1\\0&0\end{array}\right)$$
that satisfy $[H,E]=E$.

We have the one parameter subgroups  $A=\{\exp{tH}\;|\; t\in\mathbb R\}$ and $B=\{\exp{sE}\;|\; s\in\mathbb R\}$. 
The unitary dual of $G$ can be described as follows:
$$\widehat G=\{\sigma_{\pm}\}\cup\{\chi_r:r\in\mathbb R\},$$
where $\chi_r$ is a one-dimensional representation for $r\in\mathbb R$ and $\sigma_{\pm}$ are two infinite-dimensional representations defined on $L^2(\mathbb R)$ given as
\begin{eqnarray*}
\sigma_{\pm}(g)f(x)&=&\exp{(\pm i se^x)}f(x+t)\\
\chi_r(g)&=&e^{itr}
\end{eqnarray*}
for $g=\exp{sE}\exp{tH}$. Moreover, we have that
$$\left\{\begin{array}{l}
\partial\sigma_{\pm}(H)f=f'\\
\partial\sigma_{\pm}(E)f(x)=\pm ie^{x}f(x)
\end{array}\right.$$
We have the following  quasi-equivalence of the left regular representation $\lambda$:
$$\lambda\simeq \sigma_+\oplus\sigma_-,$$
and $VN(G)\simeq B(L^2(\mathbb R))\oplus B(L^2(\mathbb R))$ (see \cite[chapter 4.3]{jean_book} and \cite[chapter 3.8]{fuhr}).

For a bounded below weight function $w:\widehat B\simeq\mathbb R\to(0,\infty)$, write $M_{w^{-1}}$ for the multiplication operator on $L^2(\mathbb R)$  by the function $w^{-1}\in L^\infty(\mathbb R)$ and consider $\widetilde{M_{w^{-1}}}=\mathcal F^{-1}M_{w^{-1}}\mathcal F\in VN(\mathbb R)\simeq VN(B)$, where $\mathcal F$ is the Fourier transform. Let $\omega=\iota_B(\widetilde{M_{w^{-1}}})$ be the extended weight inverse.  Then, by \cite[\proposition~3.26 ]{gllst}, $\omega\sim (\omega(\sigma_+),\omega(\sigma_-))$, 
where $\omega(\sigma_{\pm})\xi(x)=w^{-1}(\mp e^x)\xi(x)$. 
By \theorem~\ref{lie},  $$\spec\, A(G,\omega)\simeq \{\lambda_G(g)\exp{i\partial\lambda_G(X)}\;|\;g\in G, X=sE, e^{sx}/w(x)\in L^\infty(\mathbb R)\}.$$
In particular, if $w(x)=\beta^{|x|}$, then
$$\spec\, A(G,\omega)\simeq \{\exp{(t\partial\lambda_G(H))}\exp{(s\partial\lambda_G(E))}\;|\; t\in\mathbb R, |\text{Im }s|\leq\ln\beta\}.$$
Similarly, we can start with a bounded below weight $\tilde w: \widehat A\simeq\mathbb R\to (0,\infty)$ and consider $\tilde\omega= \iota_A(\widetilde{M_{w^{-1}}})$. We have $\tilde\omega(\sigma_{\pm})=\mathcal F^{-1}\tilde w^{-1}(-x)\mathcal F$ and if $\tilde w=\tilde\beta^{|x]}$, then
\begin{eqnarray*}
\spec\, A(G,\tilde\omega)&\simeq& \{\lambda_G(g)\exp{i\partial\lambda_G(sH)}\;|\;g\in G, s\in\mathbb R, |s|\leq\ln\tilde\beta\}\\&=&\{\exp{(t\partial\lambda_G(E))}\exp{(s\partial\lambda_G(H))}\;|\; t\in\mathbb R, |\text{Im }s |\leq\ln \tilde\beta\}.
\end{eqnarray*}

\medskip

We note that by \cite{goodman} the left regular representation does not admit a dense subset of entire vectors, the fact that was an obstacle in \cite{gllst} for the study of the spectrum of $A(G,\omega)$.  The density of the set $\mathcal H_{w}(\lambda)$ of entire vectors was also important for the identification of $A(G,\omega)$ as a subset of the complexification $G_{\mathbb C}$ of $G$:  letting $\lambda_{\mathbb C}(\exp X)=\exp\partial\lambda(X)$, $X\in \mathcal g_{\mathbb C}$, one obtains a representation of $G_{\mathbb C}$ on $\mathcal H_w(\lambda)$, see \cite[\corollary~2.2]{goodman}; in general and in particular for the "$ax+b$"-group, it seems there is no natural way for $\lambda$ to be continued to a global representation of the complexified group and hence to see $G_{\mathbb C,\lambda}$ as a group. 
\end{example}


We refer the reader to \cite{gllst} for other specific examples of weights and precise descriptions of the spectrum of the associated Beurling-Fourier algebras (Examples 6.21, 7.13 and 8.22).


\section{\sc{\textbf{Some remarks and open questions}}}\label{questions}
In this section, we list open questions and make some remarks.
The most pressing question that we left unanswered is of course whether we can extend the point spectrum correspondence to general locally compact groups. Namely, 

\begin{question}
Does $\spec A(G,\omega)\subset G_{\mathbb C,\lambda}$  hold for any locally compact group $G$ and any weight inverse $\omega$?
 \end{question}
 
 Let $\Omega$ be the  $2$-cocycle, corresponding to $\omega$. 
Whether we can  answer the above question positively seems  to rely on whether the following statements are true: 

\begin{enumerate}[$(i)$]
\item $\ker \;\Omega^*=\{0\}$;
\item $S(\sigma)\sigma=S(\omega)\omega$ holds for all $\sigma\in \spec\, A(G,\omega)$;
\item $\sigma^{*}(\h)\cap \omega^{*}(\h)\ne\{0\}$ for any $\sigma\in \spec\, A(G,\omega)$;
\item any $\sigma\in \spec\, A(G,\omega)$ is a weight inverse.
\end{enumerate}

We have the implication $(ii)\Rightarrow (iii)$, as
\begin{equation}\label{cutt}
\begin{array}{cccc}
\sigma^*S(\sigma)^*\xi=\omega^*S(\omega)^*\xi\in \sigma^*(\h)\cap\omega^*(\h),&  \text{for all $\xi\in \h,$}
\end{array}
\end{equation}
and as $\ker\omega^*=\ker S(\omega)^*=\{0\}$, thus also $\ker \omega^{*}S(\omega)^{*}=\{0\}$, the subspace $\sigma^*(\h)\cap\omega^*(\h)$ is non-trivial. 
By  \corollary~\ref{spectrum}, $(i)$ and $(ii)$ give  the embedding  $\spec A(G,\omega)\subset G_{\mathbb C,\lambda}$; \theorem~\ref{mainthm} shows that $(i)$ and $(iii)$ imply $(ii)$ and $(iv)$. 
\\

As it was noticed in Section~\ref{section2} the definition of the product in $A(G,\omega)$ depends on the $2$-cocycle $\Omega$ rather than the weight inverse $\omega$, and $A(G,\omega)\simeq A(G,\Omega)$, where  $A(G,\Omega)$ is $A(G)$ (as a Banach space) with the modified product $$\begin{array}{cccc}f\cdot_\Omega g=\Gamma_{*}(\Omega (f\otimes g)), & \text{for $f,g\in A(G)$}.\end{array}$$
We note that the $2$-cocycle $\Omega$  associated with a weight inverse is always symmetric, i.e. invariant with respect to the flip automorphism on $VN(G)\bar\otimes VN(G)$, 
\begin{question}
Can one develop a similar theory for $A(G,\Omega)$ with general (symmetric) $2$-cocycle $\Omega$? What are the conditions on $\Omega$ that guarantee the existence of a weight inverse $\omega$ such that $\Gamma(\omega)\Omega=\omega\otimes\omega$?
\end{question}
 
 More specific questions are:
 \begin{question}
 For which symmetric $2$-cocycles $\Omega$  is the spectrum of $A(G,\Omega)$ non-empty?
 \end{question}

 It seems that it depends on whether or not $\Omega^{*}$, or perhaps $\Omega$, has a non-trivial kernel.
Below we give examples of $\Omega$ for which $\ker\; \Omega^{*}\neq\{0\}$ and  $\spec\, A(G,\Omega)=\emptyset$.

\begin{exa*}
Let $G=\mathbb{R},$ so that $VN(\mathbb{R})\cong L^{\infty}(\mathbb{R}).$ Let 
$$ 
\Upsilon(x)=\begin{cases} (1+x)^{x}, & \text{for $x\geq 0,$}\\ 0, & \text{for $x<0$,}\end{cases}
$$
and let  $\Omega:\mathbb{R}^{2}\to \mathbb{C}$ be the measurable function given by 
$$\Omega(x,y)=\begin{cases}
\frac{\Upsilon(x)\Upsilon(y)}{\Upsilon(x+y)}, & \text{for $x,y\geq 0,$}\\
0 ,& \text{otherwise.}
\end{cases}$$
It is easy to see that $\Upsilon(x)\Upsilon(y)\leq \Upsilon(x+y)$ for $x,y\geq 0$ and hence $\Omega(x,y)\leq 1.$
Thus $\Omega(x,y)\in L^{\infty}(\mathbb{R})\bar{\otimes}L^{\infty}(\mathbb{R})$. Moreover,
 it is not hard to see that $\Omega$ is a symmetric $2$-cocycle. Hence we have a well-defined algebra $A(\mathbb{R},\Omega).$
Using that $A(\mathbb{R})\cong L^{1}(\mathbb{R})$ via the Fourier transform, the $\Omega$-modified product between $f,g\in L^{1}(\mathbb{R})$ is
\begin{equation}\label{productin}
f*_{\tiny\Omega}g(x)={\displaystyle\int_{-\infty}^{\infty} }f(x-y)g(y)\Omega(x-y, y)\;dy=\int_{0}^{\infty}f(x-y)g(y)\frac{\Upsilon(x-y)\Upsilon(y)}{\Upsilon(x)}\; dy.
\end{equation}
Notice that if $x< 0,$ then $\Omega(x-y, y)=0$ for all $y\in \mathbb{R}$ and hence  $f*_{\Omega}g=0$ a.e. on $(-\infty, 0)$, in particular, $B:=L^1(\mathbb R^+)$ is a subalgebra of $(L^1(\mathbb R), \ast_\Omega)$. 

Next we will see that $\spec B$ is empty.
Let $B'=L^1(\mathbb R^+, \frac{1}{\Upsilon})$ with the convolution product $(f*g)(x)=\int_{0}^{\infty}f(x-y)g(y)dy$. Then $$f(x)\in B\mapsto f(x)\Upsilon(x)\in B'$$ is an isometric isomorphism.
Let $\phi$ be a linear multiplicative functional on $B'$. Then there is  $m\in L^{\infty}(\mathbb{R}^{+})$ such that  $m(x)\Upsilon(x)\in L^{\infty}(\mathbb{R}^{+})$ and 
$$
\begin{array}{ccc}
\phi(f)=\displaystyle{\int_{0}^{\infty}}m(x)f(x)\;dx, &\text{for $ f\in L^{1}(\mathbb{R}^{+})$}.
\end{array}
$$
As $\phi$ is multiplicative,  $m(x)=e^{a x}$ for some $a\in\mathbb C$. As $\lim_{x\to \infty}|e^{a x}\Upsilon(x)|=\lim_{x\to \infty}|e^{a x}(1+x)^{x}|\to \infty$ for any $a\in\mathbb{C}$,  the spectrum of $B'$ and hence of $B$ is empty. To see that this carries over to the actual algebra $A(\mathbb{R},\Omega),$ we use that $f\ast_{\Omega}g\in B$, for all $f,g\in L^{1}(\mathbb{R})$, and hence if we would have a multiplicative linear functional $\phi$ such that $\phi(f)=1$ for some $f\in L^{1}(\mathbb{R}),$ then $\phi(f\ast_{\Omega}f)=1$ and hence $\phi\in \spec B$, a contradiction.
\\

We modify the previous example slightly to obtain a continuous $2$-cocycle. Consider the function
$$\nu(x) =
\begin{cases}
e^{-\frac{1}{x}}, & \text{for $x\geq 0$,}\\
0,& \text{otherwise.}
\end{cases}
$$ It is easy to see that $\nu(x+y)\geq \nu(x)\nu(y)$ for all $x,y\in \mathbb{R}.$ Now let 
$$L(x)=\begin{cases}\nu(x)\Upsilon(x), &\text{for $x\geq 0$,}\\ 0, & \text{for $x<0.$}\end{cases} $$ and $$\Theta(x,y)=\begin{cases} \frac{L(x)L(y)}{L(x+y)}, & \text{for $x,y\geq 0,$}\\
0,& \text{otherwise,} \end{cases}$$ then $\Theta(x,y)\leq 1$ for all $x,y\in \mathbb{R}.$ Furthermore, we have $\Theta(x,y)\in C_{b}(\mathbb{R}^{2}).$ It is not that hard to see that also $A(\mathbb{R},\Theta)$ has empty spectrum (the argument is more or less the same as above). 
If $G_{\mathbb C,\lambda}\ne G$ then the  homomorphism $\bar \varphi:C_{b}(\mathbb{R})\rightarrow VN(G)$  from the proof of \proposition~\ref{weight_comp}  intertwines the coproducts and  the image $(\bar\varphi \otimes \bar\varphi)(\Theta)$ is then also a $2$-cocycle. It seems reasonable to expect that the resulting algebra would also have properties similar to the one above (i.e. not very nice spectrum-vice).
\end{exa*}

\begin{question}
What happens if we remove the condition $\ker\omega=\ker\omega^{*}=\{0\}$ from the definition of weight inverse?
\end{question}
We call such $\omega$ a partial weight inverse. A classification of partial weight inverses for discrete $G$ will be given in a separate paper.

\end{document}